\documentclass[leqno,11pt]{article}
\usepackage[utf8]{inputenc}
\usepackage[T1]{fontenc}
\usepackage{microtype}

\usepackage[letterpaper]{geometry}

\ifdefined\screenview
  \edef\mtht{\the\textheight}
  \edef\mtwd{\the\textwidth}
  \usepackage{xcolor}
  \definecolor{BackgroundColor}{RGB}{253, 246, 227}
  \pagecolor{BackgroundColor}
\fi

\usepackage{amsmath}
\usepackage{amsthm}
\usepackage{tikz-cd}
\usetikzlibrary{arrows} 
\tikzset{
  commutative diagrams/.cd, 
  arrow style=tikz, 
  diagrams={>=stealth}
}
\usetikzlibrary{matrix,decorations.pathreplacing,calc}

\usepackage{textcomp}

\usepackage[sb]{libertine}
\usepackage[varqu,varl]{zi4}
\usepackage[libertine,bigdelims,vvarbb]{newtxmath} 
\usepackage[supstfm=libertinesups,%
  supscaled=1.2,%
  raised=-.13em]{superiors}
\useosf 
\usepackage[scr=boondox,cal=euler]{mathalfa}

\usepackage{slashed}
\usepackage{esint} 
\usepackage[english]{babel}

\usepackage{imakeidx}
\makeindex[intoc]

\usepackage{csquotes}
\usepackage[
  backend=biber,
  hyperref=true,
  backref=true,
  isbn=false,
  doi=true,
  natbib=true,
  eprint=true,
  useprefix=true,
  maxcitenames=99,
  maxbibnames=99,  
  maxalphanames=99, 
  minalphanames=99,
  safeinputenc,
  style=alphabetic,
  citestyle=alphabetic,
  block=space,
  datamodel=preamble/ext-eprint
]{biblatex}
\usepackage[
  hypertexnames = false,
  colorlinks    = true,
  citecolor     = gray,
  linkcolor     = gray,
  urlcolor      = gray,
  breaklinks
]{hyperref}
\makeatletter
\DeclareFieldFormat{arxiv}{%
  arXiv\addcolon\space
  \ifhyperref
    {\href{http://arxiv.org/\abx@arxivpath/#1}{%
       \nolinkurl{#1}%
       \iffieldundef{arxivclass}
         {}
         {\addspace\texttt{\mkbibbrackets{\thefield{arxivclass}}}}}}
    {\nolinkurl{#1}
     \iffieldundef{arxivclass}
       {}
       {\addspace\texttt{\mkbibbrackets{\thefield{arxivclass}}}}}}
\makeatother
\DeclareFieldFormat{mr}{%
  MR\addcolon\space
  \ifhyperref
    {\href{http://www.ams.org/mathscinet-getitem?mr=MR#1}{\nolinkurl{#1}}}
    {\nolinkurl{#1}}}
\DeclareFieldFormat{zbl}{%
  Zbl\addcolon\space
  \ifhyperref
    {\href{http://zbmath.org/?q=an:#1}{\nolinkurl{#1}}}
    {\nolinkurl{#1}}}
\renewbibmacro*{eprint}{%
  \printfield{arxiv}%
  \newunit\newblock
  \printfield{mr}%
  \newunit\newblock
  \printfield{zbl}%
  \newunit\newblock
  \iffieldundef{eprinttype}
    {\printfield{eprint}}
    {\printfield[eprint:\strfield{eprinttype}]{eprint}}
}
\AtEveryBibitem{%
  \clearlist{publisher}%
}
\AtEveryBibitem{%
  \clearlist{address}%
}
\DeclareFieldFormat[article,inproceedings,inbook,incollection,thesis]{title}{\textit{#1}}
\renewbibmacro{in:}{}
\newcommand{\printreferences}{\printbibliography[heading=bibintoc]}

\usepackage[inline,shortlabels]{enumitem}

\usepackage{subcaption}

\usepackage[yyyymmdd]{datetime}

\usepackage{etoolbox}
\ifundef{\abstract}{}{\patchcmd{\abstract}%
    {\quotation}{\quotation\noindent\ignorespaces}{}{}}

\usepackage[super]{nth}

\usepackage{aliascnt}

\numberwithin{equation}{section}

\renewcommand{\eqref}[1]{\hyperref[#1]{\rm(\ref*{#1})}}

\def\makeautorefname#1#2{\AtBeginDocument{\expandafter\def\csname#1autorefname\endcsname{#2}}}

\newcommand{\mynewtheorem}[2]{
  \newaliascnt{#1}{equation}          
  \newtheorem{#1}[#1]{#2}
  \aliascntresetthe{#1}
  \makeautorefname{#1}{#2}
}

\mynewtheorem{axiom}{Axiom}
\newtheorem*{axiom*}{Axiom}
\mynewtheorem{theorem}{Theorem}
\newtheorem*{theorem*}{Theorem}
\mynewtheorem{prop}{Proposition}
\newtheorem*{prop*}{Proposition}

\mynewtheorem{cor}{Corollary}
\mynewtheorem{construction}{Construction}
\mynewtheorem{lemma}{Lemma}
\mynewtheorem{conjecture}{Conjecture}
\newtheorem*{conjecture*}{Conjecture}
\mynewtheorem{hyp}{Hypothesis}
\newtheorem{step}{Step}

\numberwithin{substep}{step}
\makeautorefname{step}{Step}
\makeautorefname{substep}{Step}

\numberwithin{subcase}{case}
\makeautorefname{case}{Case}
\makeautorefname{subcase}{case}

\usepackage{etoolbox}
\AtBeginEnvironment{proof}{\setcounter{step}{0}}

\theoremstyle{remark}
\mynewtheorem{remark}{Remark}
\newtheorem*{remark*}{Remark}
\mynewtheorem{convention}{Convention}
\newtheorem*{convention*}{Convention}
\newtheorem*{conventions*}{Conventions}

\theoremstyle{remark}
\mynewtheorem{warning}{Warning}

\theoremstyle{definition}
\mynewtheorem{definition}{Definition}
\newtheorem*{definition*}{Definition}
\mynewtheorem{notation}{Notation}
\mynewtheorem{data}{Data}
\mynewtheorem{example}{Example}
\newtheorem*{example*}{Example}
\mynewtheorem{exercise}{Exercise}
\mynewtheorem{solution}{Solution}
\mynewtheorem{question}{Question}
\mynewtheorem{problem}{Problem}
\newtheorem*{question*}{Question}

\mynewtheorem{summary}{Summary}

\makeautorefname{table}{Table}        
\makeautorefname{chapter}{Chapter}
\makeautorefname{section}{Section}
\makeautorefname{subsection}{Section}
\makeautorefname{subsubsection}{Section}
\makeautorefname{footnote}{Footnote}
\AtBeginDocument{\def\itemautorefname~#1\null{(#1)\null}}
\AtBeginDocument{\def\equationautorefname~#1\null{(#1)\null}}

\let\C\undefined
\let\U\undefined

\usepackage{bm}
\usepackage{mathtools} 
\usepackage{stmaryrd} 

\DeclareFontFamily{U}{mathx}{\hyphenchar\font45}
\DeclareFontShape{U}{mathx}{m}{n}{
      <5> <6> <7> <8> <9> <10>
      <10.95> <12> <14.4> <17.28> <20.74> <24.88>
      mathx10
      }{}
\DeclareSymbolFont{mathx}{U}{mathx}{m}{n}
\DeclareFontSubstitution{U}{mathx}{m}{n}
\DeclareMathAccent{\widecheck}{0}{mathx}{"71}
\DeclareMathAccent{\wideparen}{0}{mathx}{"75}

\DeclareMathOperator{\Ad}{Ad}

\DeclareMathOperator{\End}{End}

\DeclareMathOperator{\GL}{GL}

\DeclareMathOperator{\HF}{\HF}

\DeclareMathOperator{\Hom}{Hom}

\DeclareMathOperator{\Lie}{Lie}

\DeclareMathOperator{\Sym}{Sym}

\DeclareMathOperator{\coker}{coker}

\DeclareMathOperator{\im}{im}
\DeclareMathOperator{\ind}{index}

\DeclareMathOperator{\sign}{sign}

\DeclarePairedDelimiter\paren{\lparen}{\rparen}
\DeclarePairedDelimiter{\Abs}{\|}{\|}

\DeclarePairedDelimiter{\LIEbracket}{\llbracket}{\rrbracket}
\DeclarePairedDelimiter{\abs}{\lvert}{\rvert}
\DeclarePairedDelimiter{\braket}{\langle}{\rangle}
\DeclarePairedDelimiter{\liebracket}{[}{]}
\DeclarePairedDelimiter{\set}{\lbrace}{\rbrace}
\def\({\left(}
\def\){\right)}
\def\<{\left\langle}
\def\>{\right\rangle}

\newcommand{\C}{{\mathbf{C}}}
\newcommand{\Gtwo}{G_2}

\newcommand{\LIE}[2]{\LIEbracket*{#1,#2}}

\newcommand{\Met}{\sM\!et}
\newcommand{\N}{{\mathbf{N}}}

\newcommand{\PU}{{\P\U}}

\newcommand{\R}{\mathbf{R}}

\newcommand{\SO}{\mathrm{SO}}
\newcommand{\SU}{\mathrm{SU}}
\newcommand{\SW}{\mathrm{SW}}

\newcommand{\Span}[1]{\braket{#1}}
\newcommand{\Spin}{\mathrm{Spin}}
\newcommand{\Sp}{\mathrm{Sp}}

\newcommand{\U}{\mathrm{U}}

\newcommand{\Z}{\mathbf{Z}}

\newcommand{\andq}{\text{and}\quad}

\newcommand{\ch}{\mathrm{ch}}

\newcommand{\co}{\mskip0.5mu\colon\thinspace}

\newcommand{\defined}[2][\key]{\def\key{#2}\textbf{#2}\index{#1}}

\newcommand{\del}{\partial}

\newcommand{\hkred}{{/\!\! /\!\! /}}

\newcommand{\id}{\mathrm{id}}

\newcommand{\inner}[2]{\braket{#1, #2}}

\newcommand{\iso}{\cong}
\newcommand{\itref}{\eqref}
\newcommand{\lie}[2]{\liebracket*{#1,#2}}

\newcommand{\nsub}{\triangleleft}
\newcommand{\ob}{\mathrm{ob}}

\newcommand{\qandq}{\quad\text{and}\quad}

\newcommand{\qand}{\quad\text{and}}

\newcommand{\reg}{\mathrm{reg}}

\newcommand{\sfrac}[2]{\left.\(#1\)\middle/\(#2\)\right.}

\newcommand{\sw}{\fs\fw}

\newcommand{\wlie}[2]{\liebracket*{#1\wedge #2}}

\renewcommand{\H}{\mathbf{H}}
\renewcommand{\Im}{\operatorname{Im}}
\renewcommand{\O}{\mathrm{O}}
\renewcommand{\P}{\mathbf{P}}

\renewcommand{\det}{\operatorname{det}}

\renewcommand{\emptyset}{\varnothing}
\renewcommand{\epsilon}{\varepsilon}
\renewcommand{\setminus}{{\backslash}}
\renewcommand{\sp}{\mathfrak{sp}}

\renewcommand{\leq}{\leqslant}
\renewcommand{\geq}{\geqslant}

\makeatletter
\renewcommand*\env@matrix[1][*\c@MaxMatrixCols c]{%
  \hskip -\arraycolsep
  \let\@ifnextchar\new@ifnextchar
  \array{#1}}

\renewcommand\xleftrightarrow[2][]{%
  \ext@arrow 9999{\longleftrightarrowfill@}{#1}{#2}}
\newcommand\longleftrightarrowfill@{%
  \arrowfill@\leftarrow\relbar\rightarrow}
\makeatother

\newcommand{\rd}{{\rm d}}

\newcommand{\rII}{{\rm II}}

\newcommand{\bp}{{\mathbf{p}}}

\newcommand{\bA}{{\mathbf{A}}}

\newcommand{\bL}{{\mathbf{L}}}

\newcommand{\bP}{{\mathbf{P}}}
\newcommand{\bQ}{{\mathbf{Q}}}

\newcommand{\bV}{{\mathbf{V}}}

\newcommand{\cL}{\mathcal{L}}

\newcommand{\sA}{\mathscr{A}}

\newcommand{\sG}{\mathscr{G}}

\newcommand{\sI}{\mathscr{I}}

\newcommand{\sM}{\mathscr{M}}

\newcommand{\sP}{\mathscr{P}}

\newcommand{\fa}{{\mathfrak a}}

\newcommand{\fc}{{\mathfrak c}}
\newcommand{\fd}{{\mathfrak d}}
\newcommand{\fe}{{\mathfrak e}}

\newcommand{\fg}{{\mathfrak g}}

\newcommand{\fl}{{\mathfrak l}}
\newcommand{\fm}{{\mathfrak m}}
\newcommand{\fn}{{\mathfrak n}}
\newcommand{\fo}{{\mathfrak o}}

\newcommand{\fq}{{\mathfrak q}}
\newcommand{\fr}{{\mathfrak r}}
\newcommand{\fs}{{\mathfrak s}}
\newcommand{\ft}{{\mathfrak t}}
\newcommand{\fu}{{\mathfrak u}}

\newcommand{\fw}{{\mathfrak w}}
\newcommand{\fx}{{\mathfrak x}}

\newcommand{\fz}{{\mathfrak z}}

\newcommand{\fF}{{\mathfrak F}}

\newcommand{\fH}{{\mathfrak H}}

\newcommand{\fM}{{\mathfrak M}}
\newcommand{\fN}{{\mathfrak N}}

\newcommand{\fS}{{\mathfrak S}}

\newcommand{\fU}{{\mathfrak U}}

\newcommand{\fX}{{\mathfrak X}}

\newcommand{\bfS}{{\mathbf{\mathfrak S}}}

\newcommand{\slD}{\slashed D}
\newcommand{\slS}{\slashed S}

\author{
  Aleksander Doan
  \and
  Thomas Walpuski
}
\title{
  Deformation theory of the blown-up Seiberg--Witten equation in dimension three
}
\date{2018-09-26}
\begin{document}

\maketitle

\begin{abstract}
  \noindent
  Associated with every quaternionic representation of a compact, connected Lie group there is a Seiberg--Witten equation in dimension three.
  The moduli spaces of solutions to these equations are typically non-compact.
  We construct Kuranishi models around boundary points of a partially compactified moduli space.
  The Haydys correspondence identifies such boundary points with Fueter sections---solutions of a non-linear Dirac equation---of the bundle of hyperkähler quotients associated with the quaternionic representation.
  We discuss when such a Fueter section can be deformed to a solution of the Seiberg--Witten equation.
\end{abstract}


\section{Introduction}
\label{Sec_Introduction}

Associated with every quaternionic representation of a compact, connected Lie group there is a system of partial differential equations generalizing the classical Seiberg--Witten equations in dimension three and four;
see, for example, \citet{Taubes1999b}, \citet{Pidstrigach2004}, \citet{Haydys2008}, \citet[Section 6]{Salamon2013}, and \citet[Section 6(i)]{Nakajima2015}.
In fact, almost every equation studied in mathematical gauge theory arises in this way.
In the present paper we focus on the $3$--dimensional theory.
A key difficulty in studying Seiberg--Witten equations arises from the non-compactness issue caused by a lack of a priori bounds on the spinor.
This phenomenon has been studied in special cases by Taubes \cite{Taubes2012,Taubes2013,Taubes2016}, and Haydys and Walpuski \cite{Haydys2014}.
To focus on the issue of the spinor becoming very large, one passes to a blown-up Seiberg--Witten equation.
The lack of a priori bounds then manifests itself as the equation becoming degenerate elliptic when the norm of the spinor tends to infinity.
However, the Haydys correspondence allows us to reinterpret the limiting equation as a non-linear version of the Dirac equation, known as the Fueter equation \cites{Salamon2013}{Haydys2013}.
This suggests that, although formally the blown-up Seiberg--Witten equation appears to be degenerate, one should be able to develop an elliptic deformation theory around points at infinity of the moduli space.
This is what is achieved in the current paper; the main result being \autoref{Thm_KuranishiModelNearZero} below.

Our second result, \autoref{Thm_GenericOneParameterFamilies}, asserts that, under a transversality assumption, Fueter sections cause wall-crossing for the signed count of solutions to the Seiberg--Witten equation---%
a new phenomenon which has no analog in classical Seiberg--Witten theory.
In \cite{Doan2017c} we analyze this wall-crossing phenomenon for the Seiberg--Witten equation with two spinors in detail.

\citet{Donaldson2009} proposed that there should be a similar wall-crossing phenomenon for the signed count of $\Gtwo$--instantons over a $\Gtwo$--manifold.
The number of $\Gtwo$--instantons jumps due to the appearance of Fueter sections supported on $3$--dimensional associative submanifolds of the $\Gtwo$--manifold.
This is the basis of the conjectural relationship between Seiberg--Witten equations on $3$--manifolds and enumerative theories for associative submanifolds and $\Gtwo$--instantons.
Donaldson and Segal's prediction was partially confirmed in \cite{Walpuski2013a};
our \autoref{Thm_GenericOneParameterFamilies} can be understood as a Seiberg--Witten analog of this result.


\paragraph{Acknowledgements.}
This material is based upon work supported by \href{https://www.nsf.gov/awardsearch/showAward?AWD_ID=1754967&HistoricalAwards=false}{the National Science Foundation under Grant No.~1754967}
and
\href{https://sites.duke.edu/scshgap/}{the Simons Collaboration Grant on ``Special Holonomy in Geometry, Analysis and Physics''}.

\section{Main results}

For the reader's convenience, before stating our main results, we begin by reviewing the necessary background on Seiberg--Witten equations associated with quaternionic representations.

\subsection{Hyperkähler quotients of quaternionic vector spaces}

\begin{definition}
  A \defined{quaternionic Hermitian vector space} is a real vector space $S$ together with a linear map $\gamma\co \Im \H \to \End(S)$ and an inner product $\inner{\cdot}{\cdot}$ such that $\gamma$ makes $S$ into a left module over the quaternions $\H = \R\Span{1,i,j,k}$, and $i,j,k$ act by isometries.
  The \defined{unitary symplectic group} $\Sp(S)$ is the subgroup of $\GL(S)$ preserving $\gamma$ and $\inner{\cdot}{\cdot}$.
\end{definition}

Let $G$ be a compact, connected Lie group.

\begin{definition}
  A \defined{quaternionic representation} of $G$ is a Lie group homomorphism $\rho\co G \to \Sp(S)$ for some quaternionic Hermitian vector space $S$.
\end{definition}

Suppose that a quaternionic representation $\rho\co G \to \Sp(S)$ has been fixed.
By slight abuse of notation, we also denote the induced Lie algebra representation by $\rho\co \fg \to \sp(V)$.
We combine $\rho$ and $\gamma$ into the map $\bar\gamma\co \fg\otimes\Im\H \to \End(S)$ defined by
\begin{equation*}
  \bar\gamma(\xi\otimes v)\Phi \coloneq \rho(\xi)\gamma(v)\Phi.
\end{equation*}
The map $\bar\gamma$ takes values in symmetric endomorphisms of $S$.
Denote the adjoint of $\bar\gamma$ by $\bar\gamma^*\co \End(S) \to (\fg\otimes \Im\H)^*$.

\begin{prop}
  \label{Prop_CanonicalMomentMap}
  The map $\mu \co S \to (\fg \otimes \Im \H)^*$ defined by
  \begin{equation}
    \label{Eq_CanonicalMomentMap}
    \mu(\Phi) \coloneq \frac12\bar\gamma^*(\Phi\Phi^*)
  \end{equation}
  with $\Phi^* \coloneq \inner{\Phi}{\cdot}$ is a \defined{hyperkähler moment map}, that is, it is $G$--equivariant, and
  \begin{equation*}
    \inner{(d\mu)_\Phi\phi}{\xi\otimes v}
    =
      \inner{\gamma(v)\rho(\xi)\Phi}{\phi}
  \end{equation*}
  for all $\xi \in \fg$ and $v \in \Im\H$.
\end{prop}

This is a straightforward calculation. 
Nevertheless, it leads to an important conclusion: there is a hyperkähler orbifold naturally associated with the quaternionic representation.

\begin{definition}
  \label{Def_Regular}
  We call $\Phi \in S$ \defined{regular} if $(d\mu)_\Phi \co T_\Phi S \to (\fg \otimes \Im\H)^*$ is surjective.
  Denote by $S^\reg$ the open cone of regular elements of $S$.
\end{definition}

By \citet[Section 3(D)]{Hitchin1987},
\begin{equation*}
  X
  \coloneq
  S^\reg \hkred G
  \coloneq
  \(\mu^{-1}(0) \cap S^\reg\)/ G
\end{equation*}
is a hyperkähler orbifold;
see also \autoref{Prop_HyperkahlerQuotient}.
For psychological convenience, we want to make the assumption that $X$ is, in fact, a hyperkähler manifold. 
It will be important later that $X$ is a cone; that is, it carries a free $\R^+$--action.

The following summarizes the algebraic data required to write a Seiberg--Witten equation.

\begin{definition}
  \label{Def_AlgebraicData}
  A set of \defined{algebraic data} consists of:
  \begin{enumerate}
  \item
    a quaternionic Hermitian vector space $(S,\gamma,\inner{\cdot}{\cdot})$,
  \item
    a compact Lie group $H$ and a closed, connected, normal subgroup $G\nsub H$ such that $G$ acts freely on $\mu^{-1}(0)\cap S^\reg$,
  \item
    an $\Ad$--invariant inner product on $\Lie(H)$, and
  \item
    a quaternionic representation $\rho\co H \to \Sp(S)$.
  \end{enumerate}
\end{definition}

\begin{definition}
  Given a set of algebraic data as in \autoref{Def_AlgebraicData},
  the group $K \coloneq H/G$ is called the \defined{flavor symmetry group}.
\end{definition}

The groups $G$ and $K$ play different roles: $G$ is the structure group of the equation, whereas $K$ consists of any additional symmetries, which can be used to twist the setup or remain as symmetries of the theory.
On first reading, the reader should feel free to assume for simplicity that $H = G \times K$, or even that $K$ is trivial.

\subsection{The Seiberg--Witten equation}
\label{Sec_SeibergWittenEquation}

Let $M$ be a closed, connected $3$--manifold.

\begin{definition}
  \label{Def_GeometricData}
  A set of \defined{geometric data} on $M$ compatible with a set of algebraic data as in \autoref{Def_AlgebraicData} consists of:
  \begin{enumerate}
  \item
    a Riemannian metric $g$,
  \item
    a spin structure $\fs$,
  \item
     a principal $H$--bundle $Q \to M$,%
     \footnote{%
       The following observation is due to \citet[Section 3.1]{Haydys2013} and becomes important when formulating the Seiberg--Witten equation in dimension four.
       Suppose there is a homomorphism $\Z_2 \to Z(H)$ such that the non-unit in $\Z_2$ acts through $\rho$ as $-1$.
       Set $\hat H \coloneq (\Sp(1)\times H)/\Z_2$.
       All of the constructions in \autoref{Sec_SeibergWittenEquation} go through with $\fs\times Q$ replaced by a $\hat H$--principal bundle $\hat Q$.
       In the classical Seiberg--Witten theory, this corresponds to endowing the manifold with a spin$^c$ structure rather than a spin structure and a $U(1)$--bundle.
     }
     and
   \item
     a connection $B$ on the principal $K$--bundle
     \begin{equation*}
       R \coloneq Q\times_H K.
     \end{equation*}
  \end{enumerate}
\end{definition}

Suppose that a set of geometric data as in \autoref{Def_GeometricData} has been fixed.
Left-multiplication by unit quaternions defines an action $\theta\co \Sp(1) \to \O(S)$ such that
\begin{equation*}
  \theta(q) \gamma(v)\Phi = \gamma(\Ad(q)v)\theta(q)\Phi
\end{equation*}
for all $q \in \Sp(1) = \set{ q \in \H : \abs{q} = 1 }$, $v \in \Im\H$, and $\Phi \in S$.
This can be used to construct various bundles and operations as follows.

\begin{definition} 
  The \defined{spinor bundle} is the vector bundle
  \begin{equation*}
    \fS \coloneq (\fs \times Q) \times_{\Sp(1) \times H} S.
  \end{equation*}
  Since $T^*M \iso \fs \times_{\Sp(1)} \Im \H$, it inherits a \defined{Clifford multiplication} $\gamma \co T^*M \to \End(\fS)$.
\end{definition}

\begin{definition}
  Denote by $\sA(Q)$ the space of connections on $Q$.
  Set
  \begin{equation*}
    \sA_B(Q) \coloneq \set{ A \in \sA(Q) : \text{$A$ induces $B$ on $R$} }.
  \end{equation*}
  $\sA_B(Q)$ is an affine space modeled on $\Omega^1(M,\fg_P)$ with $\fg_P$ denoting the \defined{adjoint bundle} associated with $\Lie(G)$, that is,
  \begin{equation*}
    \fg_P \coloneq Q \times_{\Ad} \Lie(G).%
    \footnote{
      If $H = G \times K$, then the $G$--bundle $P$ alluded to in this notation does exist.
      In general, it does not exist but traces of it do---e.g., its adjoint bundle and its gauge group.
    }
  \end{equation*}
\end{definition}

\begin{definition}
  Every $A \in \sA_B(Q)$ defines a covariant derivative $\nabla_A \co \Gamma(\fS) \to \Omega^1(M,\fS)$.
  The \defined{Dirac operator} associated with $A$ is the linear map $\slD_A\co \Gamma(\fS) \to \Gamma(\fS)$ defined by
  \begin{equation*}
    \slD_A \Phi \coloneq \gamma(\nabla_A\Phi).
  \end{equation*}
\end{definition}

\begin{definition}  
  The hyperkähler moment map $\mu\co S \to (\Im\H\otimes\fg)^*$ induces a map
  \begin{equation*}
    \mu \co \fS \to \Lambda^2 T^*M \otimes \fg_P
  \end{equation*}
  since $(T^*M)^* \iso \Lambda^2 T^*M$.
\end{definition}

Denoting by
\begin{equation*}
  \varpi \co \fg_Q \to \fg_P
\end{equation*}
the projection induced by $\Lie(H) \to \Lie(G)$, we are finally in a position to state the equation we wish to study.

\begin{definition}
  The \defined{Seiberg--Witten equation} associated with the chosen geometric data is the following system of differential equations for $(\Phi,A) \in \Gamma(\fS) \times \sA_B(Q)$:
   \begin{equation}
    \label{Eq_SeibergWitten}
    \begin{split}
      \slD_{A}\Phi 
      &= 0 \qand \\
      \varpi F_A 
      &=  \mu(\Phi).
    \end{split}
  \end{equation}
\end{definition}

Most of the well-known equations of mathematical gauge theory on $3$-- and $4$--manifolds can be obtained as a Seiberg--Witten equation.%
\footnote{%
  In fact, if we allow the Lie groups and the representations  to be infinite-dimensional, we can also recover (special cases of) the $\Gtwo$-- and $\Spin(7)$--instanton equations \cite[Section 4.2]{Haydys2011}.
}

\begin{example}
  \label{Ex_ClassicalSeibergWitten}
  $S = \H$ and $\rho\co \U(1) \to \H$ acting by right-multiplication with $e^{i\theta}$ leads to the \defined{classical Seiberg--Witten equation} in dimension three.
\end{example}

For further examples, we refer the reader to \autoref{Sec_Examples}.

The Seiberg--Witten equation is invariant with respect to gauge transformations which preserve the flavor bundle $R$.

\begin{definition}
  The \defined{group of restricted gauge transformations}
  is
  \begin{equation*}
    \sG(P)
    \coloneq
    \set*{ u \in \sG(Q) : u \text{ acts trivially on } R }.
  \end{equation*}
\end{definition}

$\sG(P)$ is an infinite dimensional Lie group with Lie algebra $\Omega^0(M,\fg_P)$;
it acts on $\Gamma(\fS)\times\sA_B(Q)$, and preserves the space of solutions of \eqref{Eq_SeibergWitten}.

The main object of our study is the space of solutions to \eqref{Eq_SeibergWitten} modulo restricted gauge transformations.
This space depends on the geometric data chosen as in \autoref{Def_GeometricData}.
The topological part of the data, the bundles $\fs$ and $H$, will be fixed.
The remaining parameters of the equations, the metric $g$ and the connection $B$, will be allowed to vary. 
\begin{definition}
  Let $\Met(M)$ be the space of Riemannian metrics on $M$.
  The \defined{parameter space} is
  \begin{equation*}
    \sP \coloneq \Met(M)\times\sA(R).
  \end{equation*}
\end{definition}

\begin{definition}
  For $\bp = (g,B) \in \sP$, the \defined{Seiberg--Witten moduli space} is
  \begin{equation*}
    \fM_\SW(\bp)
    \coloneq
    \set*{
      [(\Phi,A)] \in \frac{\Gamma(\fS) \times \sA_B(Q)}{\sG(P)}
      :
      \begin{array}{@{}l@{}}
        (\Phi,A) \text{ satisfies } \eqref{Eq_SeibergWitten} \\
        \text{with respect to } g \text{ and } B
      \end{array}
    }.
  \end{equation*}
  The \defined{universal Seiberg--Witten moduli space} is
  \begin{equation*}
    \fM_\SW \coloneq
    \set*{
      (\bp,[(\Phi,A)]) \in \sP\times\frac{\Gamma(\fS)\times\sA(Q)}{\sG(P)}
      :
      [(\Phi, A)] \in \fM_\SW(\bp)
    }.
  \end{equation*}
\end{definition}

The Seiberg--Witten moduli spaces are endowed with the quotient topology induced from the $C^{\infty}$--topology on the spaces of connections and sections.
As we will explain in \autoref{Sec_DeformationTheoryAwayFromZero}, if $\fc_0$ is a solution of \eqref{Eq_SeibergWitten} for some $\bp_0 \in \sP$, then the deformation theory of \eqref{Eq_SeibergWitten} at $(\bp_0,\fc_0)$ is controlled by a differential graded Lie algebra (DGLA).
Associated with this DGLA is a formally-self adjoint elliptic operator $L_{\bp,\fc}$, which can be understood as a gauge fixed and co-gauge fixed linearization of \eqref{Eq_SeibergWitten}.
These operators equip $\fM_\SW$ with a real line bundle $\det L$ such that for each  $(\bp,[\fc]) \in \fM_\SW$ we have
\begin{equation*}
  (\det L)_{(\bp,[\fc])} \iso \det \ker L_{\bp,\fc} \otimes (\coker L_{\bp,\fc})^*.
\end{equation*}
The fact that the operators $L_{\bp,\fc}$ are Fredholm allows us to construct finite dimensional models of $\fM_\SW$ by standard methods.

\begin{prop}
  \label{Prop_SeibergWittenKuranishiModel}
  If $\fc_0$ is a solution of \eqref{Eq_SeibergWitten} for $\bp_0 \in \sP$ and $\fc_0$ is irreducible,%
  \footnote{%
    We say that $\fc_0$ is irreducible if  $\Gamma_{\fc_0} \coloneq \set{ u \in \sG(P) : u\fc_0 = \fc_0 } = \set{ \id }$,
    see \autoref{Def_Irreducible}.
    There is a natural generalization of \autoref{Prop_SeibergWittenKuranishiModel} to the case when $\fc_0$ is reducible. 
    Then $\Gamma_{\fc_0}$ acts on $U$ and $O$ and $\rm{ob}$ can be chosen to be $\Gamma_{\fc_0}$-equivariant, cf. \cite[Section 4.2.5]{Donaldson1990}.
    However, in this paper we focus on neighborhoods of infinity of the moduli space, and as we will see those contain only irreducible solutions.
  }
  then there is a \defined{Kuranishi model} for a neighborhood of $(\bp_0,[\fc_0]) \in \fM_\SW$; that is:
  there are an open neighborhood of $U$ of $\bp_0 \in \sP$, finite dimensional vector spaces $I$ and $O$ of the same dimension, an open neighborhood $\sI$ of $0 \in I$, a smooth map
  \begin{equation*}
    \ob \co U\times \sI \to O,
  \end{equation*}
  an open neighborhood $V$ of $(\bp_0,[\fc_0]) \in \fM_\SW$, and a homeomorphism
  \begin{equation*}
    \fx \co \ob^{-1}(0) \to V\subset \fM_\SW,
  \end{equation*}
  which maps $(\bp_0,0)$ to $(\bp_0,[\fc_0])$ and commutes with the projections to $\sP$.
  Moreover, for each $(\bp,\fc) \in \im\fx$, there is an exact sequence
  \begin{equation*}
    0 \to \ker L_{\bp,\fc} \to I
    \xrightarrow{\rd_I\ob} O \to \coker L_{\bp,\fc} \to 0
  \end{equation*}
  such that the induced maps $\det L_{\bp,\fc} \to \det(I)\otimes\det(O)^*$ define an isomorphism of line bundles $\det L \iso \fx_*(\det(I)\otimes\det(O)^*)$ on $\im \fx \subset \fM_\SW$.
\end{prop}

\subsection{The blown-up equation and the Haydys correspondence}
\label{Sec_BlownUpSeibergWitten}

Unless $\mu^{-1}(0) = \set{0}$, the projection map $\fM_\SW \to \sP$ is not expected to be proper.
This potential non-compactness phenomenon is related to the lack of a priori bounds on $\Phi$ for $(\Phi,A)$ a solution of \eqref{Eq_SeibergWitten}.
With this in mind, we blow-up the equation \eqref{Eq_SeibergWitten};
cf.~\cites[Section 2.5]{Kronheimer2007}[Equation (1.4)]{Haydys2014}.

\begin{definition}
  The \defined{blown-up Seiberg--Witten equation} is the following differential equation for $(\epsilon,\Phi,A) \in [0,\infty)\times\Gamma(\fS)\times\sA_B(Q)$:
  \begin{equation}
    \label{Eq_BlownUpSeibergWitten}
    \begin{split}
      \slD_A\Phi 
      &=
        0, \\
      \epsilon^2 \varpi F_A 
      &=
        \mu(\Phi), \qand \\
      \Abs{\Phi}_{L^2}
      &=
        1.
    \end{split}
  \end{equation}
\end{definition}

Set
\begin{equation*}
  \fS^\reg \coloneq (\fs\times Q)\times_{\Sp(1)\times H} S^\reg.
\end{equation*}

\begin{definition}
  \label{Def_PartiallyCompactifiedModuliSpace}
  The \defined{partially compactified Seiberg--Witten moduli space} is
  \begin{multline*}
    \overline{\fM}_\SW(g,B)
    \coloneq
    \Bigg\{
      (\epsilon,[(\Phi,A)]) \in [0,\infty)\times\frac{\Gamma(\fS)\times\sA_B(Q)}{\sG(P)} :
      \begin{array}{@{}l@{}}
        (\epsilon,\Phi,A) \text{ satisfies } \eqref{Eq_BlownUpSeibergWitten} \\ \text{with respect to } g \text{ and } B; \\        
        \text{if $\epsilon = 0$, then } \Phi \in \Gamma(\fS^\reg)
      \end{array}
    \Bigg\}.
  \end{multline*}
  Likewise, the \defined{universal partially compactified Seiberg--Witten moduli space} is
  \begin{equation*}
    \overline\fM_\SW \coloneq
    \set*{
      (\bp,\epsilon,[(\Phi,A)]) \in \sP\times[0,\infty)\times \frac{\Gamma(\fS)\times\sA(Q)}{\sG(P)}
      :
      (\epsilon,[(\Phi,A)]) \in \overline\fM_\SW(\bp) 
    }.
  \end{equation*}
\end{definition}

The partially compactified Seiberg--Witten moduli spaces are also naturally topological spaces.
The formal boundary of $\overline{\fM}_\SW $ is
\begin{equation*}
  \del\fM_\SW \coloneq \set*{ (\bp,0,[(\Phi,A)]) \in \overline\fM_\SW },
\end{equation*}
and the map
\begin{equation*}
  \overline\fM_\SW\setminus\del\fM_\SW \to \fM_\SW, \quad
  (\bp,\epsilon,[(\Phi,A)]) \mapsto (\bp,[(\epsilon^{-1}\Phi,A)])
\end{equation*}
is a homeomorphism.
This justifies the term ``partially compactified''.

\begin{warning}
  \label{Rem_SingularSet}  
  The space $\overline{\fM}_\SW(g,B)$ need not be compact.
  From work of \citet{Taubes2012} on \autoref{Ex_GCFlatness} with $G=\SO(3)$ and work of \citet{Haydys2014} on \autoref{Ex_ADHMSeibergWitten} with $k=1$, we expect that the actual compactification will also contain singular solutions of \autoref{Eq_BlownUpSeibergWitten} with $\epsilon = 0$;
  see \cite{Doan2017c}.
  In fact, $\del\fM_\SW$ need not be compact \cite{Walpuski2015a}.
  Precisely understanding the full compactifications is one of the central challenges in this subject.
\end{warning}

For $\epsilon = 0$, \eqref{Eq_BlownUpSeibergWitten} appears to be degenerate.
However, since $\Phi \in \Gamma(\fS^\reg)$, this equation can be understood as an elliptic PDE as follows.

\begin{definition}
  The \defined{bundle of hyperkähler quotients} $\pi\co \fX\to M$ is
  \begin{equation*}
    \fX \coloneq (\fs\times R) \times_{\Sp(1)\times K} X.
  \end{equation*}
  Its \defined{vertical tangent bundle} is
  \begin{equation*}
    V\fX \coloneq (\fs\times R) \times_{\Sp(1)\times K} TX,
  \end{equation*}
  and $\gamma \co \Im \H \to \End(S)$ induces a \defined{Clifford multiplication} $\gamma \co \pi^*TM \to \End(V\fX)$.
\end{definition}
  
\begin{definition}
  Using $B \in \sA(R)$ we can assign to each $s \in \Gamma(\fX)$ its covariant derivative $\nabla_B s \in \Omega^1(M,s^*V\fX)$.
  A section $s \in \Gamma(\fX)$ is called a \defined{Fueter section} if it satisfies the \defined{Fueter equation}
  \begin{equation}
    \label{Eq_Fueter}
    \fF(s) = \fF_B(s) \coloneq \gamma(\nabla_B s) = 0 \in \Gamma(s^*V\fX).
  \end{equation}
  The map $s \mapsto \fF(s)$ is called the \defined{Fueter operator}.%
  \footnote{%
    In the following, we will suppress the subscript $B$ from the notation.
  }
\end{definition}

An elementary but important calculation shows that a pair $(\Phi,A) \in \Gamma(\fS^\reg)\times \sA_B(Q)$ satisfies $\slD_A\Phi = 0$ and $\mu( \Phi )  = 0$ if and only if the projection $s \coloneq p\circ \Phi \in \Gamma(\fX)$ satisfies $\fF(s) = 0$.
This is part of the Haydys correspondence, which will be discussed in more detail in \autoref{Sec_FueterSections}.

The linearized Fueter operator $(\rd\fF)_s \co \Gamma(s^*V\fX) \to \Gamma(s^*V\fX)$ is a formally self-adjoint elliptic differential operator of order one.
In particular, it is Fredholm of index zero. 
However, the space of solutions to $\fF(s) = 0$, if non-empty, is never zero-dimensional.
The reason is that the hyperkähler quotient $X = S^{\reg} \hkred G$ carries a free $\R^+$--action inherited from the vector space structure on $S$. 
This induces a fiber-preserving action of $\R^+$ on $\fX$. 
One easily verifies that, for $\lambda \in \R^+$ and $s \in \Gamma(\fX)$,
\begin{equation}
  \label{Eq_FueterScaling}
  \fF( \lambda s  ) = \lambda \fF( s ).
\end{equation} 
As a result, $\R^+$ acts freely on the space of solutions to \eqref{Eq_Fueter} which shows that Fueter sections come in one-parameter families.
At the infinitesimal level, this shows that every Fueter section is obstructed.

\begin{definition}
  The \defined{radial vector field} $\hat v \in \Gamma(\fX, V \fX)$ is the vector field generating the $\R^+$--action on $\fX$.
\end{definition}
Differentiating \eqref{Eq_FueterScaling} shows that if $s$ is a Fueter section, then  $\hat v \circ s \in \Gamma(s^* V \fX)$ is a non-zero element of $\ker (\rd\fF)_s$.

\subsection{Kuranishi models for \texorpdfstring{$\overline\fM_\SW$}{the partially compactified moduli space}}

The main result of this article is the construction of Kuranishi models for $\overline\fM_\SW$ centered at points of $\del\fM_\SW$.

\begin{theorem}
  \label{Thm_KuranishiModelNearZero}
  Let $\bp_0 = (g_0,B_0) \in \sP$ and $\fc_0 = (\Phi_0,A_0) \in \Gamma(\fS^\reg)\times\sA_B(Q)$ be such that $(\bp_0,0,[\fc_0]) \in \del\fM_\SW$.
  Denote by $s_0 = p \circ \Phi_0 \in \Gamma(\fX)$ the corresponding Fueter section of $\fX$.
  Set
  \begin{equation*}
    I_\del \coloneq \ker (\rd \fF)_{s_0} \cap (\hat v\circ s)^\perp \qandq
    O \coloneq \coker (\rd \fF)_{s_0}.
  \end{equation*}
  Let $r \in \N$.
  
  There exist an open neighborhood $\sI_\del$ of $0 \in I_\del$,
  a constant $\epsilon_0 > 0$,
  an open neighborhood $U \subset \sP$ of $\bp_0$,
  a $C^{2r-1}$ map
  \begin{equation*}
    \ob \co U\times [0,\epsilon_0) \times \sI_\del \to O,
  \end{equation*}
  an open neighborhood $V$ of $(\bp_0,0,[\fc_0]) \in \overline{\fM}_\SW$, and a homeomorphism
  \begin{equation*}
    \fx\co \ob^{-1}(0) \to V \subset \overline{\fM}_\SW
  \end{equation*}
  such that the following hold:
  \begin{enumerate}
  \item
    \label{Thm_KuranishiModelNearZero_Expansion}
    There are smooth functions
    \begin{equation*}
      \ob_\del, \widehat\ob_1, \ldots, \widehat\ob_r \co U\times \sI_\del \to O
    \end{equation*}
    such that for all $m,n \in \N_0$ with $m+n \leq 2r$ we have
    \begin{equation*}
      \Abs*{\nabla_{U\times\sI_\del}^m\del_\epsilon^n\left(    
          \ob - \ob_\del - \sum\nolimits_{i=1}^r \epsilon^{2i}\widehat\ob_i
        \right)}_{C^0} = O(\epsilon^{2r-n+2}).
    \end{equation*}
  \item
    The map $\fx$ commutes with the projection to $\sP \times [0,\infty)$ and satisfies 
    \begin{equation*} 
      \fx(\bp_0,0,0) = (\bp_0,0,[\fc_0]).
    \end{equation*}
  \item
    For each $(\bp,\fc) \in \im \fx \cap \fM_\SW$, the solution $\fc$ is irreducible;
    moreover, it is unobstructed if $\rd_I\ob$ is surjective.
  \end{enumerate}
\end{theorem}

\begin{remark}
  The neighborhoods $\sI_\del$ and $U$ may depend on the choice of $r$.
\end{remark}

The difficulty in proving this theorem arises from the fact that the (gauge fixed and co-gauged fixed) linearization of \eqref{Eq_BlownUpSeibergWitten} appears to become degenerate as $\epsilon$ approaches zero.
The Haydys correspondence, however, indicates that one can reinterpret \eqref{Eq_BlownUpSeibergWitten} at $\epsilon = 0$ as the Fueter equation; in particular, as a non-degenerate elliptic PDE.
One can think of \autoref{Thm_KuranishiModelNearZero} as a gluing theorem for the Kuranishi model described in \autoref{Prop_SeibergWittenKuranishiModel} with a Kuranishi model for the moduli space of Fueter sections divided by the $\R^+$--action.

\subsection{Wall-crossing}

The main application of the work in this article---and our motivation for it---is to understand wall-crossing phenomena for signed counts of solutions to Seiberg--Witten equations arising from the non-compactness phenomenon mention in \autoref{Sec_BlownUpSeibergWitten}.
In the generic situation of \autoref{Thm_KuranishiModelNearZero}, one expects to have $\ker (\rd\fF)_{s_0} = \R\Span{\hat v\circ s_0}$.
In this case, if $\set{ \bp_t = (g_t,B_t) : t \in (-T,T)}$ is a $1$--parameter family in $\sP$, then (for $T \ll 1$) one can find a $1$--parameter family $\set{ (s_t) \in \Gamma(\fX) : t \in (-T,T)}$ of sections of $\fX$ and $\lambda \co (-T,T) \to \R$ with $\lambda(0) = 0$ such that
\begin{equation*}
  \fF_t(s_t) = \lambda(t) \cdot \hat v \circ s_t.
\end{equation*}

\begin{theorem}
  \label{Thm_GenericOneParameterFamilies}
  In the situation above and assuming $\dot\lambda(0) \neq 0$,
  for each $r \in \N$, there exist $\epsilon_0 > 0$ and
  $C^{2r-1}$ maps $t \co [0,\epsilon_0) \to (-T,T)$ and $\fc\co [0,\epsilon_0) \to \Gamma(\fS^\reg)\times \sA(Q)$ such that an open neighborhood $V$ of $(0,0,[\fc_0])$ in the parametrized Seiberg--Witten moduli space
  \begin{equation*}
    \set*{
      (t,\epsilon,[(\Phi,A)]) \in (-T,T)\times[0,\infty)\times \frac{\Gamma(\fS)\times\sA(Q)}{\sG(P)}
      :
      (\epsilon,[(\Phi,A)]) \in \overline\fM_\SW(\bp_t) 
    }
  \end{equation*}
  is given by
  \begin{equation*}
    V = \set{
      \(t(\epsilon),\epsilon,[\fc(\epsilon)]\) : \epsilon \in [0,\epsilon_0)
    }.
  \end{equation*}
  If $\fc(\epsilon) = (\Phi(\epsilon), A(\epsilon))$, then there is $\phi \in \Gamma(\fS)$ such that 
  \begin{equation*}
    \Phi(\epsilon) = \Phi_0 + \epsilon^2 \phi + O(\epsilon^4),
  \end{equation*}
  and with 
  \begin{equation*}
    \delta \coloneq \inner{\phi}{\slD_{A_0}\phi}_{L^2}
  \end{equation*}
  we have
  \begin{equation*}
    t(\epsilon) = \frac{\delta}{\dot\lambda(0)} \epsilon^4 + O(\epsilon^6).
  \end{equation*}
  For $\epsilon \in (0,\epsilon_0)$, $\fc(\epsilon)$ is irreducible;
  moreover, if $\delta \neq 0$, then $\fc(\epsilon)$ is unobstructed.
\end{theorem}

\begin{remark}
  In the situation of \autoref{Thm_GenericOneParameterFamilies}, there is no obstruction to solving the Seiberg--Witten equation to order $\epsilon^2$---in fact, a solution can be found rather explicitly.
  The obstruction to solving to order $\epsilon^4$ is precisely $\delta$.
\end{remark}

If $\fM_\SW$ is oriented (that is: $\det L \to \fM_\SW$ is trivialized) around $(\bp_0,[\fc_0])$, then identifying $\ker(\rd\fF)_{s_0} = \coker(\rd\fF)_{s_0} = \R\Span{\hat v\circ s}$ determines a sign $\sigma = \pm 1$.
If $\delta \neq 0$, then contribution of $[\fc(\epsilon)]$ should be counted with sign $-\sigma\cdot\sign(\delta)$;
as is discussed in \autoref{Sec_SeibergWittenOrientations}.
However, $\sign(\delta/\dot\lambda(0))$ also determines whether the solution $\fc(\epsilon)$ appears for $t < 0$ or $t > 0$.
Thus, the overall contributions from $\sign(\delta)$ cancel.

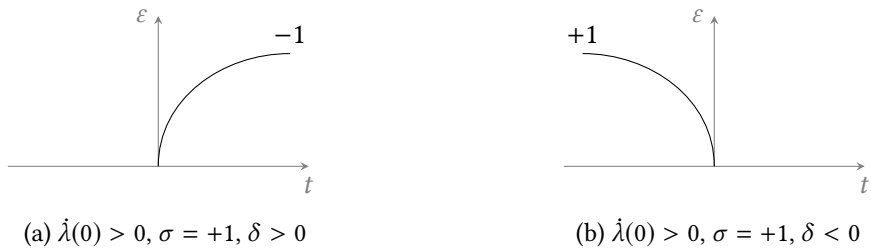
\begin{figure}[h]
  \centering
  \begin{subfigure}{.48\textwidth}
    \centering
    \begin{tikzpicture}
      \draw[-stealth,gray] (-2,0) -- (2,0) node[below]{$t$};
      \draw[-stealth,gray] (0,0) -- (0,2) node[left]{$\epsilon$};
      \draw (0,0) to [out=90,in=180] (1.75,1.5);
      \draw (1.75,1.5) node[above]{$-1$};
    \end{tikzpicture}
    \caption{$\dot\lambda(0) > 0$, $\sigma = +1$, $\delta > 0$}
    \label{Fig_TwoExamplesOfWallCrossingA}
  \end{subfigure}
  \begin{subfigure}{.48\textwidth}
    \centering
    \begin{tikzpicture}
      \draw[-stealth,gray] (-2,0) -- (2,0) node[below]{$t$};
      \draw[-stealth,gray] (0,0) -- (0,2) node[left]{$\epsilon$};
      \draw (0,0) to [out=90,in=0] (-1.75,1.5);
      \draw (-1.75,1.5) node[above]{$+1$};
    \end{tikzpicture}
    \caption{$\dot\lambda(0) > 0$, $\sigma = +1$, $\delta < 0$}
    \label{Fig_TwoExamplesOfWallCrossingB}
  \end{subfigure}
  \caption{Two examples of wall-crossing.}
  \label{Fig_TwoExamplesOfWallCrossing}
\end{figure}

This is illustrated in \autoref{Fig_TwoExamplesOfWallCrossing}, which depicts two examples of wall-crossing.
More precisely, it shows the projection of $\bigcup_{t\in(-T,T)} \overline{\fM}_\SW(\bp_t)$ on the $(t,\epsilon)$-plane.
In both cases we assume $\dot\lambda(0) > 0$ and $\sigma = +1$.
\autoref{Fig_TwoExamplesOfWallCrossingA} represents the case $\delta > 0$, in which a solution $\fc(\epsilon)$ with sign $\sign(\fc(\epsilon)) = -\sigma\cdot \sign(\delta) = -1$ is born at $t = 0$.
\autoref{Fig_TwoExamplesOfWallCrossingB} represents the case $\delta < 0$, in which $\sign(\fc(\epsilon)) = +1$ and the solution dies at $t = 0$.
In both cases, as we cross from $t < 0$ to $t > 0$ the signed count of solutions to the Seiberg--Witten equation changes by $-1$.


\section{Deformation theory of the Seiberg--Witten equation}
\label{Sec_DeformationTheoryAwayFromZero}

We begin with the deformation theory of the blown-up Seiberg--Witten equation away from $\epsilon = 0$, that is, with the deformation theory of the Seiberg--Witten equation itself.
All of this material is standard, but it will set the stage for what is to come.

\subsection{The Seiberg--Witten DGLA}

The deformation theory of the Seiberg--Witten equation is controlled by the following differential graded Lie algebra (DGLA).

\begin{definition}
  \label{Def_SWDGLA_Differential}
  Denote by $L^\bullet$ the graded real vector space given by
  \begin{align*}
    L^0 &\coloneq \Omega^0(M,\fg_P), \\
    L^1 &\coloneq \Gamma(\fS)\oplus\Omega^1(M,\fg_P), \\
    L^2 &\coloneq \Gamma(\fS)\oplus\Omega^2(M,\fg_P), \qand \\
    L^3 &\coloneq \Omega^3(M,\fg_P).
  \end{align*}
  Denote by $\LIE{\cdot}{\cdot} \co L^\bullet \otimes L^\bullet \to L^\bullet$ the graded skew-symmetric bilinear map defined by
  \begin{align*}
    \LIE{a}{b}
    &\coloneq
      \wlie{a}{a} && \text{for } a,b \in \Omega^\bullet(M,\fg_P), \\
    \LIE{\xi}{\phi}
    &\coloneq
      \rho(\xi)\phi &&\text{for } \xi \in \Omega^0(M,\fg_P) \text{ and } \phi \in \Gamma(\fS) \text{ in degree $1$ or $2$}, \\
    \LIE{a}{\phi}
    &\coloneq
      -\bar\gamma(a)\phi &&\text{for } a \in \Omega^1(M,\fg_P) \text{ and } \phi \in \Gamma(\fS) \text{ in degree $1$}, \\
    \LIE{\phi}{\psi}
    &\coloneq
      -2\mu(\phi,\psi) &&\text{for } \phi, \psi \in \Gamma(\fS) \text{ in degree $1$, and} \\
    \LIE{\phi}{\psi}
    &\coloneq
      -*\rho^*(\phi\psi^*) &&\text{for } \phi \in \Gamma(\fS) \text{ in degree $1$ and } \psi \in \Gamma(\fS) \text{ in degree 2}.
  \end{align*}
  Given $\fc = (\Phi,A) \in \Gamma(\fS)\times\sA_B(Q)$, define the degree one linear map $\delta^\bullet = \delta_\fc^\bullet \co L^\bullet \to L^{\bullet+1}$ by
  \begin{align*}
    \delta_\fc^0(\xi)
    &\coloneq
      \begin{pmatrix}
        -\rho(\xi)\Phi \\
        \rd_A\xi
      \end{pmatrix}, \\
    \delta_\fc^1(\phi,a)
    &\coloneq
      \begin{pmatrix}
        -\slD_A\phi - \bar\gamma(a)\Phi \\
        -2\mu(\Phi,\phi) + \rd_Aa
      \end{pmatrix}, \qandq \\
    \delta_\fc^2(\psi,b)
    &\coloneq
      *\rho^*(\psi\Phi^* ) + \rd_Ab.
  \end{align*}
\end{definition}

\begin{prop}
  \label{Prop_SWDGLA}
  The algebraic structures defined in \autoref{Def_SWDGLA_Differential} determine a DGLA which controls the deformation theory of the Seiberg--Witten equation; that is:
  ~
  \begin{enumerate}
  \item
    \label{Prop_SWDGLA_Lie}
    $(L^\bullet,\LIE{\cdot}{\cdot})$ is a graded Lie algebra.
  \item
    \label{Prop_SWDGLA_CompatibleWithD}
    If $\fc = (\Phi,A)$ is a solution of \eqref{Eq_SeibergWitten}, then $(L^\bullet,\LIE{\cdot}{\cdot},\delta_\fc^\bullet)$ is a DGLA.
  \item
    \label{Prop_SWDGLA_MaurerCartan}
    Suppose that $\fc = (\Phi,A)$ is a solution of \eqref{Eq_SeibergWitten}.
    For every $\hat\fc = (\phi,a) \in L^1$, $(\Phi+\phi,A+a)$ solves \eqref{Eq_SeibergWitten} if and only if it is a \defined{Maurer--Cartan element}, that is,
    $\delta_\fc \hat\fc + \frac12\LIE{\hat\fc}{\hat\fc} = 0$.
  \end{enumerate}
\end{prop}

The verification of \eqref{Prop_SWDGLA_Lie} and \eqref{Prop_SWDGLA_CompatibleWithD} is somewhat lengthy, and is deferred to \autoref{Sec_DGLA}.
Part \eqref{Prop_SWDGLA_MaurerCartan}, however, is straightforward.

\begin{definition}
  \label{Def_Irreducible}
  Let $\fc \in \Gamma(\fS)\times\sA_B(Q)$ be a solution of \eqref{Eq_SeibergWitten}.
  We call
  \begin{equation*}
    \Gamma_{\fc} \coloneq \set{ u \in \sG(P) : u\fc = \fc }
  \end{equation*}
  the group of \defined{automorphisms} of $\fc$.
  Its Lie algebra is the cohomology group  $H^0(L^\bullet,\delta_\fc)$; 
  $H^1(L^\bullet,\delta_\fc)$ is the space of \defined{infinitesimal deformations}, and $H^2(L^\bullet,\delta_\fc)$ the space of \defined{infinitesimal obstructions}.
  We say that $\fc$ is \defined{irreducible} if $\Gamma_{\fc} = 0$,
  and \defined{unobstructed} if $H^2(L^\bullet,\delta_\fc) = 0$.
\end{definition}

\begin{remark}
  $H^3(L^\bullet,\delta_\fc)$ has no immediate interpretation, but it is isomorphic to $H^0(L^\bullet,\delta_\fc)$, since the complex $(L^\bullet,\delta_\fc)$ is self-dual (up to signs).
  The latter also shows that $H^1(L^\bullet,\delta_\fc)$ is isomorphic to $H^2(L^\bullet,\delta_\fc)$.
\end{remark}

\subsection{The linearized Seiberg--Witten equation}
\label{Sec_LinearizedSeibergWittenEquation}

The operators
\begin{align*}
  \tilde\delta_\fc^0 
  &\coloneq
    \delta_\fc^0 \co \Omega^0(M,\fg_P) \to \Gamma(\fS)\oplus\Omega^1(M,\fg_P), \\
  \tilde\delta_\fc^1 
  &\coloneq
    (\id_\fS\oplus *)\circ\delta_\fc^1 \co \Gamma(\fS)\oplus\Omega^1(M,\fg_P) \to \Gamma(\fS)\oplus\Omega^1(M,\fg_P), \qand \\
  \tilde\delta_\fc^2 
  &\coloneq
    -* \circ ~\delta_\fc^2\circ (\id_\fS\oplus *) \co \Gamma(\fS)\oplus\Omega^1(M,\fg_P) \to \Omega^0(M,\fg_P)
\end{align*}
satisfy
\begin{equation*}
  (\tilde\delta_\fc^0)^* = \delta_\fc^2 \qandq
  (\delta_\fc^1)^* = \delta_\fc^1,
\end{equation*}
and $L_\fc \co \Gamma(\fS) \oplus \Omega^1(M,\fg_P)\oplus \Omega^0(M,\fg_P) \to  \Gamma(\fS) \oplus \Omega^1(M,\fg_P)\oplus \Omega^0(M,\fg_P)$ defined by
\begin{align*}
  L_\fc
  &\coloneq
    \begin{pmatrix}
      \tilde\delta_\fc^1 & \tilde\delta_\fc^0 \\
      \tilde\delta_\fc^2 & 0
    \end{pmatrix} \\
  &=
  \begin{pmatrix}
    -\slD_A & 0 & 0 \\
    0 & *\rd_A & \rd_A \\
    0 & \rd_A^* & 0
  \end{pmatrix}
  +
  \begin{pmatrix}
    0 & -\bar\gamma(\cdot)\Phi & -\rho(\cdot)\Phi \\
    -2*\mu(\Phi,\cdot) & 0 & 0 \\
    -\rho^*(\cdot\,\Phi^*) & 0 & 0
  \end{pmatrix}
\end{align*}
is formally self-adjoint and elliptic.

\begin{definition}
  We call $L_\fc$ the \defined{linearization }of the Seiberg--Witten equation at $\fc$.
\end{definition}

If $\fc$ is a solution of \eqref{Eq_SeibergWitten}, then Hodge theory identifies $H^1(L^\bullet,\delta_\fc)\oplus H^0(L^\bullet,\delta_\fc)$ with $\ker L_\fc$ and $H^2(L^\bullet,\delta_\fc)\oplus H^3(L^\bullet,\delta_\fc)$ with $\coker L_\fc$.
The fact that $(L^\bullet,\delta_\fc)$ is self-dual (up to signs) manifests itself as $L_\fc$ being formally self-adjoint.
After gauge fixing and co-gauge fixing, we can understand \eqref{Eq_SeibergWitten} as an elliptic PDE as follows.

\begin{prop}
  \label{Prop_ExtendedSeibergWitten}
  Given
  \begin{equation*}
    \fc_0 = (\Phi_0,A_0) \in \Gamma(\fS)\times\sA_B(Q),
  \end{equation*}
  define $Q\co \Gamma(\fS)\oplus\Omega^1(M,\fg_P)\oplus\Omega^0(M,\fg_P) \to \Gamma(\fS)\oplus\Omega^1(M,\fg_P)\oplus\Omega^0(M,\fg_P)$
  by
  \begin{equation*}
    Q(\phi,a,\xi)
    \coloneq
    \begin{pmatrix}
      -\bar\gamma(a)\phi \\
      \frac12*[a\wedge a] - *\mu(\phi) \\
      0
    \end{pmatrix},
  \end{equation*}
  $\fe_{\fc_0} \in \Gamma(\fS)\oplus\Omega^1(M,\fg_P)\oplus\Omega^0(M,\fg_P)$ by
  \begin{equation*}
    \fe_{\fc_0}
    \coloneq
    \begin{pmatrix}
      -\slD_{A_0} \Phi_0 \\
      *\varpi F_{A_0} - *\mu(\Phi_0) \\
      0
    \end{pmatrix},
  \end{equation*}
  and set
  \begin{equation*}
    \sw_{\fc_0}(\hat\fc) \coloneq L_{\fc_0}\hat \fc + Q_{\fc_0}(\hat\fc) + \fe_{\fc_0}.
  \end{equation*}

  There is a constant $\sigma > 0$ depending on $\fc_0$ such that,
  for every
  $\hat\fc = (\phi,a,\xi) \in \Gamma(\fS)\oplus\Omega^1(M,\fg_P)\oplus\Omega^0(M,\fg_P)$
  satisfying $\Abs{(\phi,a)}_{L^\infty} < \sigma$,
  the equation
  \begin{equation*}
    \sw_{\fc_0}(\hat\fc) = 0
  \end{equation*}
  holds if and only if
  $\fc_0 + (\phi,a)$ satisfies \eqref{Eq_SeibergWitten} and the gauge fixing condition
  \begin{equation}
    \label{Eq_SeibergWittenGaugeFixing}
    \rd_{A_0}^* a - \rho^*(\phi\Phi_0^*) = 0
  \end{equation}
  as well as
  \begin{equation*}
    \rd_{A_0}\xi = 0 \qandq \rho(\xi)\Phi_0 = 0;
  \end{equation*}
  moreover, if $\fc_0$ is infinitesimally irreducible (that is: $H^0(L^\bullet,\delta_{\fc_0}) = 0$), then $\xi = 0$.
\end{prop}

The proof requires a number of useful identities for $\mu$ which are summarized and proved in \autoref{Sec_UsefulFormulae}.

\begin{proof}
  Setting $\Phi \coloneq \Phi_0 + \phi$ and $A \coloneq A_0 + a$, the equation $\sw_{\fc_0}(\hat\fc) = 0$ amounts to
  \begin{align*}
    \slD_A\Phi + \rho(\xi)\Phi_0 &= 0, \\
    \varpi F_A + *\rd_{A_0}\xi &= \mu(\Phi), \qand \\
    \rd_{A_0}^* a - \rho^*(\phi\Phi_0^*) &= 0.
  \end{align*}
  Since
  \begin{equation*}
    \rd_A\mu(\Phi) = -*\rho^*\((\slD_A\Phi)\Phi^*\)
  \end{equation*}
  by \eqref{Eq_DMu}, applying $\rd_A$ to the second equation above and using the first equation we obtain
  \begin{equation*}
    \rd_{A_0}^*\rd_{A_0}\xi
    + \rho^*\((\rho(\xi)\Phi_0)\Phi_0^*\)
    - *[a\wedge*\rd_{A_0}\xi]
    + \rho^*\((\rho(\xi)\Phi_0)\phi^*\)
    = 0.
  \end{equation*}
  Taking the $L^2$ inner product with $\xi_0$, the component of $\xi$ in the $L^2$ orthogonal complement of $\ker\delta_{\fc_0}$ and integrating by parts yields that
  \begin{equation*}
    \Abs{\rd_{A_0}\xi}_{L^2}^2
    + \Abs{\rho(\xi)\Phi_0}_{L^2}^2
    =
      \inner{*[a\wedge*\rd_{A_0}\xi]}{\xi_0}_{L^2}
      - \inner{\rho(\xi)\Phi_0}{\rho(\xi_0)\phi}_{L^2}.
  \end{equation*}
  The right-hand side can be bounded by a constant $c > 0$ (depending on $\fc_0$) times
  \begin{equation*}
    \Abs{(a,\phi)}_{L^\infty}
    \(\Abs{\rd_{A_0}^*\xi}_{L^2}^2
    + \Abs{\rho(\xi)\Phi_0}_{L^2}^2\).
  \end{equation*}
  Therefore, if $\Abs{(a,\phi)}_{L^\infty} < \sigma \coloneq 1/c$, then
  \begin{equation*}
    \rd_{A_0}\xi = 0 \qandq
    \rho(\xi)\Phi_0 = 0.
  \end{equation*}
  It follows that $\hat\fc + (\phi,a)$ satisfies \eqref{Eq_SeibergWitten}.

  Since $\xi \in H^0(L^\bullet,\delta_{\fc_0})$, it vanishes if $\fc_0$ infinitesimally irreducible.
\end{proof}

The following standard observation shows that imposing the gauge fixing condition \eqref{Eq_SeibergWittenGaugeFixing} is mostly harmless, as long as we are only interested in small variations $\hat\fc$; c.f.~\cite[Proposition 4.2.9]{Donaldson1990}.

\begin{notation}
  In what follows we denote by $W^{k,p}\Gamma(\fS)$ the space of sections of $\fS$ of Sobolev class $W^{k,p}$.
  We use similar notations for spaces of connections, gauge transformations, and differential forms.
\end{notation}

\begin{prop}
  \label{Prop_SeibergWittenGaugeFixing}
  Fix $k \in \N$ and $p \in (1,\infty)$ with $(k+1)p > 3$. 
  Given
  \begin{equation*}
    \fc_0 = (\Phi_0,A_0) \in W^{k+1,p}\Gamma(\fS) \times W^{k+2,p}\sA_B(Q),
  \end{equation*} 
  there is a constant $\sigma>0$ such that if we set
  \begin{equation*}
    \fU_{\fc_0,\sigma}
    \coloneq
    \set*{
      \hat \fc \in B_\sigma(0) \subset W^{k+1,p}\Gamma(\fS) \times W^{k+2,p}\Omega^1(M,\fg_P)
      :
      \rd_{A_0}^* a - \rho^*(\phi\Phi_0^*) = 0
    },
  \end{equation*}
  then the map
  \begin{equation*}
    \fU_{\fc_0,\sigma}/\Gamma_{\fc_0} \ni
    [\hat\fc] 
    \mapsto
    [\fc_0+\hat\fc] \in
    \frac{W^{k+1,p}\Gamma(\fS) \times W^{k+2,p}\sA_B(Q)}{W^{k+3,p}\sG(P)}
  \end{equation*}
  is a homeomorphism onto its image;
  moreover, $\Gamma_{\fc_0+\hat\fc}$ is the stabilizer of $\hat\fc$ in $\Gamma_\fc$.
\end{prop}

For $\hat\fc = (\phi,a,\xi)$ and $(\Phi,A) = \fc = \fc_0 + (\phi,a)$, we have
\begin{equation*}
  (\rd\sw_{\fc_0})_{\hat\fc}
  =
  \begin{pmatrix}
    -\slD_A & 0 & 0 \\
    0 & *\rd_A & \rd_{A_0} \\
    0 & \rd_{A_0}^* & 0
  \end{pmatrix}
  +
  \begin{pmatrix}
    0 & -\bar\gamma(\cdot)\Phi & -\rho(\cdot)\Phi_0 \\
    -2*\mu(\Phi,\cdot) & 0 & 0 \\
    -\rho^*(\cdot\,\Phi_0^*) & 0 & 0
  \end{pmatrix}.
\end{equation*}
In particular, $(\rd\sw_{\fc_0})_{0}$ agrees with $L_{\fc_0}$.
The following proposition explains the relation between $(\rd\sw_{\fc_0})_{\hat\fc}$ and $L_{\fc}$ for $\fc = (\Phi,A,0) + \hat\fc$.

\begin{prop}
  \label{Prop_DSWvsL}
  In the situation of \autoref{Prop_SeibergWittenGaugeFixing}, if $\hat\fc \in \fU_{\fc_0,\sigma}$ and $\fc = \fc_0 + \hat\fc$,
  then there is a $\tau > 0$ and a smooth map $\phi_{\fc_0,\fc}\co B_\tau(\fc) \to B_\sigma(0)$ which maps $\fU_{\fc,\tau}$ to $\fU_{\fc_0,\sigma}$, commutes with the projection to $\sfrac{W^{k+1,p}\Gamma(\fS) \times W^{k+2,p}\sA_B(Q)}{W^{k+3,p}\sG(P)}$, and satisfies
  \begin{equation*}
    (\rd\phi)_{\fc}^{-1} (\rd\sw_{\fc_0})_{\hat\fc} (\rd\phi)_{\fc} = (\rd\sw_{\fc})_0 = L_\fc.
  \end{equation*}
\end{prop}

\subsection{Construction of Kuranishi models}

The method of the proof of \autoref{Prop_SeibergWittenKuranishiModel} is quite standard, c.f.~\cite[Section 4.2]{Donaldson1990}.
Fix $k \in \N$ and $p \in (1,\infty)$ with $(k+1)p > 3$.
Given $\bp = (g,B) \in \sP$, set
\begin{equation*}
  \fM_\SW^{k,p}(\bp)
  \coloneq 
  \set*{
    [(\Phi,A)] \in \frac{W^{k+1,p}\Gamma(\fS) \times W^{k+2,p}\sA_B(Q)}{W^{k+3,p}\sG(P)}
    :
    \begin{array}{@{}l@{}}
        (\Phi,A) \text{ satisfies } \eqref{Eq_SeibergWitten} \\
      \text{with respect to } g \text{ and } B
    \end{array}
  },
\end{equation*}
and define $\fM_\SW^{k,p}$ accordingly.
It is a consequence of elliptic regularity for $L_\fc$ and \autoref{Prop_SeibergWittenGaugeFixing}, that the inclusion $\fM_\SW \subset \fM_\SW^{k,p}$ is a homeomorphism.
This together with \autoref{Prop_ExtendedSeibergWitten} and \autoref{Prop_SeibergWittenGaugeFixing} implies that if $(\bp_0,[\hat\fc_0]) \in \fM_\SW$ is irreducible,
then there is a constant $\sigma > 0$ and an open neighborhood $U$ of $\bp \in \sP$ such that 
if $B_\sigma(0)$ denotes the open ball of radius $\sigma$ centered at $0$ in $    W^{k+1,p}\Gamma(\fS) \oplus W^{k+2,p}\Omega^1(M,\fg_P) \oplus W^{k+2,p}\Omega^0(M,\fg_P)$, then
\begin{equation*}
  \set{
    (\bp,\hat \fc) \in U\times B_\sigma(0)
    :
    \sw_{\bp,\fc_0}(\hat\fc) = 0
  } \ni (\bp,[(\phi,a,\xi)])
  \mapsto
  (\bp,[\fc+(\phi,a)]) \in \fM_\SW
\end{equation*}
is a homeomorphism onto its image.
Here we use subscripts to denote the dependence of $L_{\fc_0}$, $Q$, $\fe_{\fc_0}$, and $\sw_{\fc_0}$ on the parameter $\bp \in \sP$.
The proof of \autoref{Prop_SeibergWittenKuranishiModel} is completed by applying the following result to $\sw_{\bp,\fc_0}$ with $I = \ker L_{\bp_0,\fc_0}$ and $O = \coker L_{\bp_0,\fc_0}$.

\begin{lemma}
  \label{Lem_ZeroSetsOfFredholmMaps}
  Let $X$ and $Y$ be Banach spaces,
  let $U \subset X$ be a neighborhood of $0 \in X$,
  let $P$ be a Banach manifold, and
  let $F \co P\times U \to Y$ be a smooth map of the form
  \begin{equation*}
    F(p,x) = L(p,x) + Q(p,x) + \fe(p)
  \end{equation*}
  such that:
  \begin{enumerate}
  \item 
    $L$ is smooth, for each $p \in P$, $L_p \coloneq L(p,\cdot) \co X \to Y$ is a Fredholm operator, and we have $\sup_{p \in P} \Abs{L_p}_{\cL(X,Y)} < \infty$,
  \item
    $Q$ is smooth and there exists a $c_Q > 0$ such that, for all $x_1,x_2 \in X$ and all $p \in P$, we have
    \begin{equation}
      \label{Eq_QQuadraticEstimate}
      \Abs{Q(x_1,p)-Q(x_2,p)}_Y \leq c_Q\(\Abs{x_1}_X + \Abs{x_2}_X\)\Abs{x_1-x_2}_X,
    \end{equation}
    and
  \item
    $\fe \co P \to Y$ is smooth and there is a constant $c_\fe$ such that $\Abs{\fe}_Y \leq c_\fe$.
  \end{enumerate}

  Let $I \subset X$ be a finite dimensional subspace and let $\pi\co X \to I$ be a projection onto $I$.
  Let $O \subset Y$ be a finite dimensional subspace, let $\Pi\co Y \to O$ be a projection onto $O$, and denote by $\iota \co O \to Y$ the inclusion.
  Suppose that, for all $p \in P$, the operator $\bar L_p \co O \oplus X \to I \oplus Y$ defined by
  \begin{equation*}
    \bar L_p
    \coloneq
    \begin{pmatrix}
      0 & \pi \\
      \iota & L_p
    \end{pmatrix}
  \end{equation*}
  is invertible, and suppose that $c_R \coloneq \sup_{p \in P} \Abs{\bar L_p^{-1}}_{\cL(Y,X)} < \infty$.

  If $c_\fe \ll_{c_R,c_Q} 1$, then there is an open neighborhood $\sI$ of $0 \in I$,
  an open subset $V \subset P\times U$ containing $P\times \set{0}$, and
  a smooth map
  \begin{equation*}
    x \co P\times \sI \to X
  \end{equation*}
  such that, for each $(p,x_0) \in \sI \times P$, $(p,x(p,x_0))$ is the unique solution $(p,x)\in V$ of
  \begin{equation}
    \label{Eq_F=0}
    (\id_Y-\Pi) F(p,x) = 0 \qandq \pi x = x_0.
  \end{equation}
  In particular, if we define $\ob \co P\times\sI \to O$ by
  \begin{equation*}
    \ob(p,x_0) \coloneq \Pi F(p,x(p,x_0)),
  \end{equation*}
  then the map $\ob^{-1}(0) \to F^{-1}(0) \cap V$ defined by
  \begin{equation*}
    (p,x_0) \mapsto (p,x(p,x_0))
  \end{equation*}
  is a homeomorphism.
  Moreover, for every $(p,x_0) \in P\times\sI$ and $x = x(p,x_0)$, 
  we have an exact sequence
  \begin{equation*}
    0 \to \ker \del_x F(p,x) \to I \xrightarrow{\del_{x_0}\ob(p,x_0)} O \to \coker \del_xF(p,x) \to 0;
  \end{equation*}
  which induces an isomorphism $\det \del_xF \iso \det I\otimes(\det O)^*$.
\end{lemma}
 
\begin{proof}[Proof sketch]
  This is result is essentially a summary of the discussion in \citet[Section 5]{Guo2013}; see also \cite[Proposition 4.2.4]{Donaldson1990}.
  The crucial point is that $\bar L_p$ induces an inverse to $(\id_Y-\Pi)L_p \co \ker\pi \to \ker\Pi$;
  thus by the Inverse Function Theorem there are $\sigma,\tau > 0$ such that $U' \coloneq B_\sigma(0)\times B_\tau(0) \subset I \times \ker \pi$, and there is a smooth map $\Xi\co P\times U' \to \ker \pi$ such that, for each $p \in P$ and $x \in B_\sigma(0)$:
  \begin{enumerate}
  \item
    $\Xi(p,x_0,0) = 0$,
  \item
    $\Xi(p,x_0,\cdot)$ is a diffeomorphism onto its image, and
  \item
    for all $p \in P$ and $(x_0,x_1) \in U'$, we have
    \begin{equation*}
      \tilde F(p,x_0,x_1)
      =
      F(p,x_0,\Xi(p,x_0,x_1))
      =
      \begin{pmatrix}
        f(p,x_0,x_1) \\
        G_p (x_0)
      \end{pmatrix}
      + \fe(p)
    \end{equation*}
    where $G_p\co \ker\pi \to \ker\Pi$ is the linear isomorphism induced by $\bar L_p$ and $f(p,0,0) = 0$.
  \end{enumerate}
  If $c_\fe \ll 1$, then $G_p^{-1}(\id_Y-\Pi)\fe(p) \in B_\tau(0)$ and we can take
  \begin{equation*}
    \sI = B_\sigma(0) \qandq
    x(p,x_0) \coloneq \paren[\big]{x_0, G_p^{-1}(\id_Y-\Pi)\fe(p)}.
  \end{equation*}

  We have
  \begin{equation*}
    \ker \del_xF \iso \ker \del_x \tilde F
    \qandq
    \coker \del_xF \iso \coker \del_x \tilde F.
  \end{equation*}
  However, $\del_x\tilde F$ induces $G_p(x_0)$ from $\ker \pi$ to $\ker \Pi$.
  Therefore,
  \begin{equation*}
    \ker \del_x \tilde F \iso \ker \del_{x_0}f \qandq
    \coker \del_x \tilde F \iso \coker \del_{x_0}f.
  \end{equation*}
  Since
  \begin{equation*}
    \ob(p,x_0) = f(p,x_0,G_p^{-1}(\id_Y-\Pi)\fe(p)),
  \end{equation*}
  it follows that
  \begin{equation*}
    \ker \del_xF \iso \ker \del_{x_0}\ob 
    \qandq
    \coker \del_xF \iso \coker \del_{x_0}\ob.
    \qedhere
  \end{equation*}
\end{proof}

\subsection{Orientations}
\label{Sec_SeibergWittenOrientations}

For the purpose of counting solutions to \eqref{Eq_SeibergWitten} orientations play an important role.
Suppose a trivialization $\tau\co \det L \iso \R$ has been chosen.
If $\bp \in \sP$ and $[\fc] \in \fM_\SW(\bp)$ is irreducible and unobstructed, then $\det L_\fc = \det(0) \otimes \det(0)^* = \R\otimes \R^*$ is canonically trivial, and we define $\tau([\fc]) = +1$ if the isomorphism $\tau_{[\fc]} \co \R \iso \R$ is orientation preserving and $\tau(\fc) = -1$ if it is orientation reversing.
If $\bp_0 \in \sP$ is such that all $[\fc] \in \fM_\SW(\bp_0)$ are irreducible and unobstructed, and $\fM_\SW(\bp_0)$ is finite, then we can define
\begin{equation*}
  n_\SW(\bp_0) \coloneq \sum_{[\fc] \in \fM_\SW(\bp_0)} \tau([\fc]).
\end{equation*}

The following is a useful criterion to check whether $\det L$ can be trivialized.

\begin{prop}
  \label{Prop_Orientations}
  Suppose that algebraic data as in \autoref{Def_AlgebraicData} and compatible geometric data as in \autoref{Def_GeometricData} have been fixed.
  Let $\rho_G \co G \to \Sp(S)$ be the restriction of the quaternionic representation $\rho \co H \to \Sp(S)$ to $G \nsub H$.
  Denote by $c_2 \in B\Sp(S)$ the universal second Chern class.
  If $(B\rho_G)^*c_2 \in H^4(BG,\Z)$ can be written as
  \begin{equation}
  \label{Eq_CharacteristicClasses}
    (B\rho_G)^*c_2 = 2x + \alpha_1 y_1^2 + \cdots + \alpha_k y_k^2
  \end{equation}
  with $x \in H^4(BG,\Z)$, $y_1,\ldots,y_k \in H^2(BG,\Z)$, and $\alpha_1, \ldots, \alpha_k \in \Z$, then
  \begin{equation*}
    \det L \to \sP\times \frac{\Gamma(\fS)\times \sA(Q)}{\sG(P)}
  \end{equation*}
  is trivial.
\end{prop}

\begin{proof}
  The parameter space $\sP$ is contractible;
  hence, it is enough to fix an element $\bp \in \sP$ and prove that $\det L$ is trivial over the second factor. 
  We need to show that if $(\fc_t)_{t\in[0,1]}$ is a path in $\Gamma(\fS)\times \sA_B(Q)$ and $u \in \sG(P)$ is such that $u\fc_1 = \fc_0$, then the spectral flow of $(L_{\fc_t})_{t\in[0,1]}$ is even.
  The mapping torus of $u \co Q \to Q$ is a principal $H$--bundle $\bQ$ over $S^1 \times M$, and the path $(\fc_t)_{t\in[0,1]}$ induces a connection $\bA$ on $\bQ$.
  Over $S^1 \times M$ we also have an adjoint bundle $\fg_\bP$ and the spinor bundles $\bfS^+$ and $\bfS^-$ associated with $\bQ$ via the quaternionic representation $\rho \co H \to \Sp(S)$.
  According to Atiyah--Singer--Patodi, the spectral flow of $(L_{\fc_t})_{t\in[0,1]}$ is the index of the operator $\bL = \partial_t - L_{\fc_t}$ which can be identified with an operator 
  \begin{equation*}
    \bL \co \Gamma(\bfS^+)\oplus\Omega^1(S^1\times M,\fg_{\bP})\to \Gamma(\bfS^-)\oplus\Omega^+(S^1\times M,\fg_{\bP})\oplus\Omega^0(S^1\times M,\fg_{\bP}).
  \end{equation*}
  In our case, $\bL$ is homotopic through Fredholm operators to the sum of the Dirac operator $\slD^+_{\bA} \co \Gamma(\bfS^+) \to \Gamma(\bfS^-) $ and the Atiyah--Hitchin--Singer operator $\rd^+_\bA \oplus \rd^*_\bA$ for $\fg_\bP$.
  The index of the Atiyah--Hitchin--Singer operator is $-2p_1(\fg_\bP)$ and thus even.
  To compute the index of the Dirac operator, observe that the vector bundle $\bV \coloneq \bQ \times_{\rho} S$ inherits from $S$ the structure of a left-module over $\H$ and that
  \begin{equation*}
    \bfS^{\pm} = \slS^{\pm} \otimes_{\H} \bV,
  \end{equation*}
  where $\slS^{\pm}$ are the usual spinor bundles of $S^1 \times M$ with the spin structure induced from that on $M$ and we use the structure of $\slS^{\pm}$ as a right-modules over $\H$.
  $\bfS^{\pm}$ is a real vector bundle: it is a real form of $\slS^{\pm} \otimes_{\C} \bV$.
  Therefore, the complexification of $\slD_{\bA}^+$ is the standard complex Dirac operator on $\slS^{\pm}$ twisted by $\bV$.
  By the Atiyah--Singer Index Theorem, 
  \begin{align*}
    \ind \slD_{\bA}^+ 
    &= 
      \int_{S^1\times M} \hat A(S^1 \times M)\ch(\bV) \\
    &=
      \int_{S^1 \times M} \ch_2( \bV)
     =
      - \int_{S^1 \times M} c_2(\bV).
  \end{align*}
  The classifying map $f_\bV \co S^1\times M \to B\Sp(S)$ of $\bV$ is related to the classifying map $f_\bQ \co S^1\times M \to BG$ of $\bQ$ through
  \begin{equation*}
    f_\bV = B\rho_G \circ f_\bQ,
  \end{equation*}
  and
  \begin{equation*}
    c_2(\bV)
    = f_\bV^*c_2
    = f_\bQ^*(B\rho_G)^*c_2.
  \end{equation*}
  Since the intersection form of $S^1\times M$ is even, the hypothesis implies that the right-hand side of the above index formula is even.
\end{proof}

\begin{remark}
  If $G$ is simply--connected,
  then the condition \eqref{Eq_CharacteristicClasses} is satisfied if and only if the image of
  \begin{equation*}
    (\rho_G)_* \co \pi_3( G )\to \pi_3( \Sp(S) ) = \Z
  \end{equation*}
  is generated by an even integer. 
  To see this, observe that $BG$ is $3$--connected;
  hence, by the Hurewicz theorem, $H_4(BG,\Z) = \pi_4(BG) \cong \pi_3(G)$ and $H_i( BG, \Z) = 0$ for $i = 1,2,3$.
  The same is true for $\Sp(S)$, and we have a commutative diagram
  \begin{equation*}
    \begin{tikzcd}
      H_4(BG,\Z) \ar[r,"(B\rho_G)_*"] \ar[d,"\iso"] & H_4(B\Sp(S),\Z) \ar[d,"\iso"] \\
      \pi_3(G) \ar[r,"(\rho_G)_*"] & \pi_3(\Sp(S)).
    \end{tikzcd}
  \end{equation*}
  The group $H_4(BG,\Z)$ is freely generated by some elements $x_1, \ldots, x_k$.
  Let $x^1, \ldots, x^k$ be the dual basis of $H^4(BG,\Z) = \Hom(H_4(BG,\Z), \Z)$.
  Likewise, $H_4(B\Sp(S),\Z)$ is freely generated by the unique element $z$ satisfying $\inner{c_2}{z} = 1$. 
  We have
  \begin{equation}
   \label{Eq_Pullback}
   (B \rho_G)^* c_2 = \sum_{i=1}^k \inner{ (B \rho_G)^* c_2 }{x_i} x^i 
  \end{equation}
  and 
  \begin{equation*}
    \inner{ (B \rho_G)^* c_2 }{x_i} = \inner{ c_2}{ (B \rho_G)_* x_i}.
  \end{equation*}
  Therefore, the coefficients in the sum \eqref{Eq_Pullback} are all even if and only if the image of $(B \rho_G)_*$ is generated by $2m z$ for some $m \in \Z$.
\end{remark}

\begin{example}
  The hypothesis of \autoref{Prop_Orientations} holds when $S = \H \otimes_{\C} W$ for some complex Hermitian vector space $W$ of dimension $n$ and $\rho_G$ is induced from a unitary representation $G \to \U(W)$;
  as is the case for the representations leading to the classical Seiberg--Witten and $\U(n)$--monopole equations, see \autoref{Ex_ClassicalSeibergWitten} and \autoref{Ex_NonabelianMonopoles}.
  To see that $(B\rho_G)^*c_2$ is of the desired form, note that if $E$ is a rank $n$ Hermitian vector bundle, then the corresponding quaternionic Hermitian bundle obtained via the inclusion $\U(n) \to \Sp(n)$ is $\H \otimes_\C E = E \oplus \bar E$ and
  \begin{equation*}
    c_2(\H\otimes_\C E) = c_2(E\oplus\bar E) = 2c_2(E) - c_1(E)^2.
  \end{equation*}
\end{example}

\begin{example}
  The hypothesis of \autoref{Prop_Orientations} is also satisfied when $S = \H \otimes_{\R} W$ for a real Euclidean vector space $W$,
  and $\rho_G$ is induced from an orthogonal representation $G \to \SO(W)$;
  as is the case for the equation for flat $G^\C$--connections,
  see \autoref{Ex_GCFlatness}.
  To see that $(B\rho_G)^*c_2$ is of the desired form, note that if $E$ is a Euclidean vector bundle of rank $n$,
  then the associated quaternionic Hermitian vector bundle is $\H\otimes_\R E$ and
  \begin{equation*}
    c_2(\H \otimes_{\R} E) = -2p_1(E).
  \end{equation*}
\end{example}

If two quaternionic representations satisfy the hypothesis of \autoref{Prop_Orientations}, then so does their direct sum.
Therefore, the previous two examples together show that $\det L$ is trivial for the ADHM Seiberg--Witten equation described in \autoref{Ex_ADHMSeibergWitten}.

\begin{example}
  In general, $\det L$ need not be orientable.
  If $S = \H$ and $G = H = \Sp(1)$ acts on $S$ by right multiplication, then it is easy to see that the gauge transformation of the trivial bundle $Q = S^3 \times \SU(2)$ induced by $S^3 \iso \SU(2)$ gives rise to an odd spectral flow.
\end{example}


\section{The Haydys correspondence}
\label{Sec_FueterSections}

In order to discuss the deformation theory \emph{on} the boundary of $\overline\fM_\SW$, it will be helpful to review the correspondence, discovered by \citet[Section 4.1]{Haydys2011}, between Fueter sections of $\fX$ and solutions $(\Phi,A) \in \Gamma(\fS^\reg)\times\sA_B(Q)$ of
\begin{equation}
  \label{Eq_LimitingSeibergWitten}
  \slD_A \Phi = 0 \qandq \mu(\Phi) = 0.
\end{equation}

For what follows it will be important to recall some details of hyperkähler reduction construction.

\begin{prop}[{\citet[Section 3(D)]{Hitchin1987}}]
  \label{Prop_HyperkahlerQuotient}
  If $\rho\co G \to \Sp(S)$ is a quaternionic representation, then the following hold:
  \begin{enumerate}
  \item 
    \label{Prop_HyperkahlerQuotient_X}
    The space
    \begin{equation*}
      X
      \coloneq
        S^\reg \hkred G
      \coloneq
        \(\mu^{-1}(0) \cap S^\reg\)/ G
    \end{equation*}
    is an orbifold (with discrete isotropy groups).
  \item
    \label{Prop_HyperkahlerQuotient_HorizontalNormal}
    Denote by $p \co \mu^{-1}(0) \cap S^\reg \to X$ the canonical projection.
    Set
    \begin{align*}
      \fH
      &\coloneq
        (\ker \rd p)^\perp \cap T (\mu^{-1}(0)\cap S^\reg) 
        \qand
      \\
      \fN
      &\coloneq
        \fH^\perp \subset TS|_{\mu^{-1}(0)\cap S^\reg}.
    \end{align*}
    For each $\Phi \in \mu^{-1}(0) \cap S^\reg$, $(\rd p)_\Phi \co \fH_\Phi \to T_{[\Phi]} X$ is an isomorphism, and
    \begin{equation}
      \label{Eq_NPhi}
      \fN_\Phi
      =
        \im \(\rho(\cdot)\Phi \oplus \bar\gamma(\cdot)\Phi\co \fg \otimes \H \to S\).
    \end{equation}
  \item
    \label{Prop_HyperkahlerQuotient_HorizontalNormalQuaternionic}
    For each $\Phi \in \mu^{-1}(0) \cap S^\reg$, $\gamma$ preserves the splitting $S = \fH_\Phi \oplus \fN_\Phi$.
  \item
    \label{Prop_HyperkahlerQuotient_G+GammaX}
    There exist a Riemannian metric $g_X$ on $X$ and a Clifford multiplication 
    \begin{equation*}
      \gamma_X\co \Im \H \to \End(TX)
    \end{equation*}
    such that
    \begin{equation*}
      p^*g_X = \inner{\cdot}{\cdot} \qandq
      p^*\gamma_X = \gamma.
    \end{equation*}
  \item
    \label{Prop_HyperkahlerQuotient_GammaParallel}
    $\gamma_X$ is parallel with respect to $g_X$; hence, $X$ is a hyperkähler orbifold---which is called the \defined{hyperkähler quotient} of $S$ by $G$.
  \end{enumerate}
\end{prop}

\begin{remark}
  More generally, $\mu^{-1}(0)/G$ can be decomposed into a union of hyperkähler manifolds according to the conjugacy class of the stabilizers in $G$;
  see \citet[Theorem 2.1]{Dancer1997}.
\end{remark}

\subsection{Lifting sections of \texorpdfstring{$\fX$}{X}}

\begin{prop}
  \label{Prop_HaydysCorrespondence}
  Given a set of geometric data as in \autoref{Def_GeometricData}, set
  \begin{equation*}
    X \coloneq S^\reg\hkred G
    \qandq
    \fX \coloneq (\fs\times R) \times_{\Sp(1)\times K} X.
  \end{equation*}
  Denote by $p \co S^\reg\cap\mu^{-1}(0) \to X$ the canonical projection.
  \begin{enumerate}
  \item
    \label{Prop_HaydysCorrespondence_ExistenceOfLift}
    If $s \in \Gamma(\fX)$, then there exist a principal $H$--bundle $Q$ together with an isomorphism $Q\times_H K \iso R$ and a section $\Phi \in \Gamma(\fS^\reg)$ of
    \begin{equation*}
      \fS^\reg \coloneq (\fs\times Q) \times_{\Sp(1)\times H} S^\reg
    \end{equation*}
    satisfying
    \begin{equation*}
      \mu(\Phi) = 0 \qandq
      s = p \circ \Phi.
    \end{equation*}
    $Q$ and $Q\times_H K \iso R$ are unique up to isomorphism, and every two lifts $\Phi$ are related by a unique gauge transformation in $\sG(P)$.
  \item
    \label{Prop_HaydysCorrespondence_HorizontalConnection}
    Suppose $B \in \sA(R)$.
    If $\Phi \in \Gamma(\fS^\reg)$ satisfies $\mu(\Phi) = 0$, then there is a unique $A \in \sA_B(Q)$ such that $\nabla_A\Phi \in \Omega^1(M,\fH_\Phi)$.
    In particular, for this connection
    \begin{equation*}
      p_*(\slD_A\Phi) = \fF(s).
    \end{equation*}
  \item
    \label{Prop_HaydysCorrespondence_HorizontalConnection2}
    The condition $p_*(\slD_A\Phi) = \fF(s)$ characterizes $A \in \sA_B(Q)$ uniquely.
  \end{enumerate}
\end{prop}

\begin{proof}
  Part \eqref{Prop_HaydysCorrespondence_ExistenceOfLift} is proved by observing that the lifts exists locally and that the obstruction to the local lifts patching defines a cocycle which determines $Q$; see \cite{Haydys2011} for details.

  We prove \eqref{Prop_HaydysCorrespondence_HorizontalConnection}.
  For an arbitrary connection $A_0 \in \sA_B(Q)$ and for all $x \in M$, we have
  \begin{equation*}
    (\nabla_{A_0}\Phi)(x) \in T^*_xM \otimes T_{\Phi(x)}(S^\reg\cap\mu^{-1}(0)).
  \end{equation*}
  By \autoref{Prop_HyperkahlerQuotient}\eqref{Prop_HyperkahlerQuotient_HorizontalNormal} there exists a unique $a \in \Omega^1(M,\fg_P)$ such that
  \begin{equation*}
    \nabla_{A_0+a} \Phi \in \Omega^1(M,\fH_\Phi).
  \end{equation*}
  The assertion in \eqref{Prop_HaydysCorrespondence_HorizontalConnection} now follows from the fact that for $s = p\circ\Phi$ we have $p_*(\nabla_{A_0}\Phi) = \nabla_B s$ and the definitions of $\slD_A$ and $\fF$.

  We prove \eqref{Prop_HaydysCorrespondence_HorizontalConnection2}.
  If $a \in \Omega^1(M,\fg_P)$ and $A+a$ also satisfies this condition, then we must have
  \begin{equation*}
    \bar\gamma(a)\Phi = 0.
  \end{equation*}
  This is impossible because $\Phi \in \Gamma(\fS^{\reg})$, that is, $(\rd\mu)_\Phi$ is surjective; hence, its adjoint $\bar\gamma(\cdot)\Phi$ is injective.
\end{proof}

\begin{prop}
  \label{Prop_GaugeFixingLifts}
  Given a set of geometric data as in \autoref{Def_GeometricData},
  set
  \begin{equation*}
    R \coloneq Q\times_HK, \quad
    \fX \coloneq (\fs\times R)\times_{\Sp(1)\times K}X, \qandq
    \fS^\reg \coloneq (\fs\times Q)\times_{\Sp(1)\times H}S^\reg.
  \end{equation*}
  The map
  \begin{equation*}
    \Gamma(\mu^{-1}(0)\cap\fS^\reg)/\sG(P) \to \Gamma(\fX)
  \end{equation*}
  \begin{equation*}
    [\Phi] \mapsto p \circ \Phi
  \end{equation*}
  is a homeomorphism onto its image.
\end{prop}

\begin{proof}
  Fix $\Phi_0 \in \Gamma(\mu^{-1}(0)\cap\fS^\reg)$ and set $s_0 \coloneq p\circ\Phi_0 \in \Gamma(\fX)$.
  Given $0 < \sigma \ll 1$, for every $\Phi \in \Gamma(\mu^{-1}(0)\cap\fS^\reg)$ with $\Abs{\Phi-\Phi_0}_{L^\infty} < \sigma$, there is a unique $u \in \sG(P)$ such that
  \begin{equation*}
    u\Phi \perp \im\paren[\big]{\rho(\cdot)\Phi_0 \co \Gamma(\fg_P) \to \Gamma(\fS)};
  \end{equation*}
  moreover, for every $k \in \N$,
  \begin{equation*}
    \Abs{u\Phi - \Phi_0}_{C^k} \lesssim_k \Abs{\Phi - \Phi_0}_{C^k}.
  \end{equation*}
  Thus, it suffices to show that the map
  \begin{equation*}
    \set*{
      \Phi \in \Gamma(\mu^{-1}(0)\cap \fS^\reg)
      :
      \begin{array}{@{}l@{}}
        \Abs{\Phi-\Phi_0}_{L^\infty} < \sigma \text{ and} \\
        \Phi \perp \im\paren[\big]{\rho(\cdot)\Phi_0 \co \Gamma(\fg_P) \to \Gamma(\fS)}
      \end{array}
    }
    \to \Gamma(\fX)
  \end{equation*}
  is a homeomorphism onto its image.
  This, however, is immediate from the Implicit Function Theorem and the fact that the tangent space at $\Phi_0$ to the former space is $\Gamma(\fH_{\Phi_0})$ and the derivative of this map is the canonical isomorphism $\Gamma(\fH_{\Phi_0}) \iso \Gamma(s_0^*V\fX)$ from \autoref{Prop_HyperkahlerQuotient}\eqref{Prop_HyperkahlerQuotient_HorizontalNormal}.
\end{proof}

In the situation of \autoref{Prop_HaydysCorrespondence},
we have $\abs{\Phi} = \abs{\hat v \circ s}$.
The preceding results thus imply the following.

\begin{cor}
  \label{Cor_HaydysCorrespondence}
  Let $R$ be a principal $K$--bundle.
  Set $\fX \coloneq R\times_K X$ and
  \begin{equation*}
    \fM_F
    \coloneq
      \set*{ (\bp,s) \in \sP \times \Gamma(\fX) : \fF(s) = 0 \text{ and } \Abs{\hat v\circ s}_{L^2} = 1 }.
  \end{equation*}
  The map 
  \begin{equation*}
    \coprod_{Q} \del\fM_{\SW,Q} \to \fM_F
  \end{equation*}
  defined by
  \begin{equation*}
    (\bp,[(\Phi,A)]) \mapsto (\bp,p\circ\Phi)
  \end{equation*}
  is a homeomorphism.
  Here, the disjoint union is taken over all isomorphism classes of principal $H$--bundles $Q$ with isomorphisms $Q\times_HK \iso R$.
\end{cor}

\subsection{Lifting infinitesimal deformations}

\begin{prop}
  \label{Prop_InfinitesimalHaydysCorrespondence}
  For $\Phi \in \Gamma(\mu^{-1}(0)\cap\fS^\reg)$, set $s \coloneq p\circ\Phi$ and let $A \in \sA_B(Q)$ be as in \autoref{Prop_HaydysCorrespondence}.
  The isomorphism $p_* \co \Gamma(\fH_\Phi) \to \Gamma(s^*V\fX)$ identifies $\pi_\fH\nabla_A \co \Omega^0(M,\fH_\Phi) \to \Omega^1(M,\fH_\Phi)$ with $\nabla_B \co \Omega^0(M,s^*V\fX) \to \Omega^1(M,s^*V\fX)$.
\end{prop}

\begin{proof}
  If $(\Phi_t)$ is a one-parameter family of local sections of $\mu^{-1}(0)\cap\fS^\reg$ with
  \begin{equation*}
    (\del_t \Phi_t)|_{t=0} = \phi,
  \end{equation*}
  $A_t$ are as in \autoref{Prop_HaydysCorrespondence}, and $a = (\del_t A_t)|_{t=0}$, then we have
  \begin{equation*}
    \left.\del_t\(\pi_{\fH_{\Phi_t}}\nabla_{A_t} \Phi_t\)\right|_{t=0}
    = 
    \left.(\del_t \pi_{\fH_{\Phi_t}})\right|_{t=0} \nabla_{A_0} \Phi_0
    + \pi_{\fH_{\Phi_0}} (\rho(a)\Phi_0)
    + \pi_{\fH_{\Phi_0}} (\nabla_{A_0}\phi).
  \end{equation*}
  The first term vanishes because $\nabla_{A_0} \Phi_0 \in \Gamma(\fH_{\Phi_0})$, and the second term vanishes because of \autoref{Prop_HyperkahlerQuotient}\eqref{Prop_HyperkahlerQuotient_HorizontalNormal}.
\end{proof}

If $\Phi \in \Gamma(\mu^{-1}(0)\cap\fS^\reg)$,
then the induced splitting $\fS = \fH_\Phi\oplus \fN_\Phi$ given by \autoref{Prop_HyperkahlerQuotient}\eqref{Prop_HyperkahlerQuotient_HorizontalNormal} need not be parallel for $A$ as in \autoref{Prop_HaydysCorrespondence}.

\begin{definition}
  The \defined{second fundamental forms} of the splitting $\fH_\Phi\oplus \fN_\Phi$ are defined by
  \begin{align*}
    \rII
    &\coloneq
      \pi_\fN \nabla_A \in \Omega^1(M,\Hom(\fH_\Phi,\fN_\Phi))
    \qand \\
    \rII^*
    &\coloneq
      -\pi_\fH \nabla_A \in \Omega^1(M,\Hom(\fN_\Phi,\fH_\Phi)).
  \end{align*}
\end{definition}

We decompose the Dirac operator $\slD_A$ according to $\fS = \fH_\Phi\oplus\fN_\Phi$ as
\begin{equation}
  \label{Eq_DiracInBlockForm}
  \slD_{A} = 
  \begin{pmatrix}
    \slD_\fH & -\gamma \rII^* \\
    \gamma \rII & \slD_\fN
  \end{pmatrix}
\end{equation}
with
\begin{align*}
  \slD_\fH
  &\coloneq
    \gamma(\pi_\fH \nabla_A) \co \Gamma(\fH_\Phi) \to \Gamma(\fH_\Phi)
  \qand \\
  \slD_\fN
  &\coloneq
    \gamma(\pi_\fN \nabla_A) \co \Gamma(\fN_\Phi) \to \Gamma(\fN_\Phi).
\end{align*}

The following result helps to better understand the off-diagonal terms in \eqref{Eq_DiracInBlockForm}.

\begin{prop}
  \label{Prop_GammaII}
  Suppose $\Phi \in \Gamma(\mu^{-1}(0)\cap\fS^\reg)$ and $\slD_A\Phi = 0$.
  Writing $\phi \in \Gamma(\fN_\Phi)$ as
  \begin{equation*}
    \phi = \rho(\xi)\Phi + \bar\gamma(a)\Phi
  \end{equation*}
  for $\xi \in \Gamma(\fg_P)$ and $a \in \Omega^1(M,\fg_P)$, we have
  \begin{equation*}
    -\gamma\rII^*\phi
    =
      2\sum_{i=1}^3 \pi_{\fH}\(\rho(a(e_i))\nabla^A_{e_i}\Phi\).
  \end{equation*}
  Here $(e_1,e_2,e_3)$ is a local orthonormal frame.
\end{prop}

\begin{proof}
  Since $\nabla\Phi \in \Omega^1(M,\fH_\Phi)$ and $\slD_A\Phi = 0$,
  we have
  \begin{align*}
    -\gamma\rII^* (\rho(\xi)\Phi + \bar\gamma(a)\Phi_0)
    &=
      \sum_{i=1}^3
      \gamma(e^i)\pi_\fH\(\rho(\xi)\nabla^A_{e_i}\Phi + \bar\gamma(a)\nabla^A_{e_i}\Phi\) \\
    &=
      \sum_{i=1}^3
      \pi_\fH\((\gamma(e^i)\bar\gamma(a)+\bar\gamma(a)\gamma(e^i))\nabla^A_{e_i}\Phi\) \\
    &=
      2\sum_{i=1}^3
      \pi_{\fH}(\rho(a(e_i))\nabla^A_{e_i}\Phi).
      \qedhere
  \end{align*}
\end{proof}

\begin{prop}
  \label{Prop_slDfH=dfF}
  The isomorphism $p_* \co \Gamma(\fH_\Phi) \to \Gamma(s^* V \fX)$ identifies the linearized Fueter operator $(\rd\fF)_s \co \Gamma(s^*V\fX) \to \Gamma(s^*V\fX)$ with $\slD_\fH \co \Gamma(\fH_\Phi) \to \Gamma(\fH_\Phi)$. 
\end{prop}

\begin{proof}
  The linearized Fueter operator is given by
  \begin{equation*}
    (\rd\fF)_s\hat s = \gamma(\nabla_B \hat s)
  \end{equation*}
  The assertion thus follows from \autoref{Prop_HyperkahlerQuotient}\eqref{Prop_HyperkahlerQuotient_G+GammaX} and \autoref{Prop_InfinitesimalHaydysCorrespondence}.
\end{proof}


\section{Deformation theory of Fueter sections}
\label{Sec_FueterDeformationTheory}

\begin{prop}
  \label{Prop_FueterKuranishiModel}
  Let $s_0 \in \Gamma(\fX)$ be a Fueter section with respect to $\bp_0 = (g_0,B_0) \in \sP$.
  Denote by $\fc_0 \in \Gamma(\fS^\reg)\times\sA(P)$ a lift of $s_0$.
  There exist an open neighbourhood $U$ of $\bp_0 \in \sP$, an open neighborhood
  \begin{equation*}
    \sI_\del \subset I_\del \coloneq \ker (\rd \fF)_{s_0} \cap (\hat v\circ s)^\perp
  \end{equation*}
  of $0$, a smooth map
  \begin{equation*}
    \ob_\del \co U \times \sI_\del \to \coker(\rd\fF)_{s_0},
  \end{equation*}
  an open neighborhood $V$ of $([\bp_0,\fc_0]) \in \del \fM_\SW$, and a homeomorphism
  \begin{equation*}
    \fx_\del \co \ob_\del^{-1}(0) \to V \subset \del\fM_\SW
  \end{equation*}
  which maps $(\bp_0,0)$ to $(\bp_0,0,[\fc_0])$ and commutes with the projections to $\sP$.
\end{prop}

Since $\del\fM_\SW \iso \fM_F$ through the Haydys correspondence, this has a straightforward proof using \autoref{Lem_ZeroSetsOfFredholmMaps}, which makes no reference to the Seiberg--Witten equation.
However, this is not the approach we take because our principal goal is to compare the deformation theory of Fueter sections with that of solutions of the Seiberg--Witten equation.

Fix $k \in \N$ and $p \in (1,\infty)$ with $(k+1)p > 3$.
Let
\begin{equation*}
  \del\fM_\SW^{k,p} =
  \set*{
    (\bp,[(\Phi,A)]) \in \sP\times \frac{W^{k+1,p}\Gamma(\fS)\times W^{k,p}\sA(Q)}{W^{k+1,p}\sG(P)}
    :
    \begin{array}{@{}l@{}}
      A \text{ induces } B, \\ 
      (\Phi,A) \text{ satisfies } \eqref{Eq_LimitingSeibergWitten}, \\
      \text{and }  \| \Phi \|_{L^2} = 1
    \end{array}
  }.
\end{equation*}
By the Haydys correspondence $\del\fM_\SW^{k,p}$ is homeomorphic to $\fM_F^{k,p}$, the universal moduli space of normalized $W^{k+1,p}$ Fueter sections of $\fX$.
Consequently, for $\ell \in \N$ and $q \in (1,\infty)$ with $\ell \geq k$ and $q \geq p$, the inclusions $\del\fM_\SW^{\ell,q} \subset \del\fM_\SW^{k,p} \subset \del\fM_\SW$ are homeomorphisms;
see also \autoref{Prop_RegularityAtEpsilon=0}.

\begin{prop}
  \label{Prop_L0}
  Assume the situation of \autoref{Prop_FueterKuranishiModel}.
  For $\bp \in \sP$, set
  \begin{align*}
    X_0 &\coloneq  W^{k+1,p}\Gamma(\fS) \oplus W^{k,p}\Omega^1(M,\fg_P)\oplus W^{k,p}\Omega^0(M,\fg_P) \\
    \andq
    Y &\coloneq  W^{k,p}\Gamma(\fS) \oplus W^{k+1,p}\Omega^1(M,\fg_P)\oplus W^{k+1,p}\Omega^0(M,\fg_P)\oplus\R,
  \end{align*}
  and define a linear map $L_{\bp,0} \co X_0 \to Y$, a quadratic map $Q_{\bp,0} \co X_0 \to Y$, and $\fe_{\bp,0} \in Y$ by
  \begin{gather*}
    L_{\bp,0}
    \coloneq
    \begin{pmatrix}
      -\slD_{A_0} & -\bar\gamma(\cdot)\Phi_0 & -\rho(\cdot)\Phi_0 \\
      -2*\mu(\Phi_0,\cdot) & 0 & 0 \\
      -\rho^*(\cdot\,\Phi_0^*) & 0 & 0\\
      2\inner{\cdot}{\Phi_0}_{L^2} & 0 & 0
    \end{pmatrix}, \\
    Q_{\bp,0}(\phi,a,\xi)
    \coloneq
    \begin{pmatrix}
      -\bar\gamma(a)\phi \\
      -*\mu(\phi) \\
      0 \\
      \Abs{\phi}_{L^2}^2
    \end{pmatrix},
    \qandq
    \fe_{\bp,0} \coloneq
    \begin{pmatrix}
      -\slD_{A_0} \Phi_0 \\
      -\mu(\Phi_0) \\
      0 \\
      \Abs{\Phi_0}_{L^2}^2 - 1
    \end{pmatrix},
  \end{gather*}
  respectively.%
  \footnote{%
    The term $\fe_{\bp,0}$ vanishes for $\bp = \bp_0$.
  }
  
  There exist a neighborhood $U$ of $\bp_0 \in \sP$ and $\sigma > 0$, such that, for every $\bp \in U$ and $\hat\fc = (\phi,a,\xi) \in B_\sigma(0) \subset X_0$, we have
  \begin{equation}
    \label{Eq_L0}
    L_{\bp,0}\hat\fc + Q_{\bp,0}(\hat\fc) + \fe_{\bp,0} = 0
  \end{equation}
  if and only if $\xi = 0$ and $(\Phi,A) = (\Phi_0 + \phi,A_0 + a)$ satisfies 
  \begin{equation}
    \label{Eq_SeibergWittenLimit}
    \slD_A \Phi = 0 \qandq
    \mu(\Phi) = 0
  \end{equation}  
  as well as
  \begin{equation*}
    \| \Phi \|_{L^2} = 1 \qandq 
    \rho^*(\Phi\Phi_0^*) = 0.
  \end{equation*}
\end{prop}

\begin{remark}
  The above proposition engages in the following abuse of notation.
  If $A_0 \in \sA_B(Q)$ and $B' \in \sA(R)$, then $b = B'-B \in \Omega^1(M,\fg_R)$.
  Since $\Lie(K) = \Lie(G)^\perp \subset \Lie(H)$ we have a map $\Omega^1(M,\fg_R) \to \Omega^1(M,\fg_Q)$ and can identify $A_0 \in \sA_B(Q)$ with ``$A_0$'' $= A_0 + b \in \sA_{B'}(Q)$.
\end{remark}

Together with (the argument from the proof of) \autoref{Prop_GaugeFixingLifts} we obtain the following.

\begin{cor}
  Assume the situation of \autoref{Prop_FueterKuranishiModel}.
  With $U \subset \sP$ and $\sigma > 0$ as in \autoref{Prop_L0}, the map
  \begin{equation*}
    \set*{
      (\bp,\hat\fc) \in U\times B_\sigma(0)
      \text{ satisfying }
      \eqref{Eq_L0}
    }
    \to \del\fM_\SW
  \end{equation*}
  defined by
  \begin{equation*}
     (\bp,\phi,a,\xi) \mapsto (\bp,[(\Phi_0+\phi,A_0+a)])
  \end{equation*}
  is a homeomorphism onto a neighborhood of $[\fc_0]$.
\end{cor}

\begin{proof}[Proof of \autoref{Prop_L0}]
  If $\hat\fc = (\phi,a,\xi)$ satisfies \eqref{Eq_L0}, then $\Phi=\Phi_0+\phi$ and $A = A_0 + a$ satisfy
  \begin{equation*}
    \slD_{A}\Phi + \rho(\xi)\Phi_0 = 0, \quad \mu(\Phi) = 0, \qandq
    \rho^*(\phi\Phi_0^*) = 0.
  \end{equation*}
  Hence, by \autoref{Prop_DMu_D*Mu},
  \begin{equation*}
    0
    =
    \rd_A\mu(\Phi)
    = -\rho(\slD_A\Phi\Phi^*)
    = \rho^*(\rho(\xi)\Phi_0(\Phi_0+\phi))
    = R_{\Phi_0}^*R_{\Phi_0}\xi + O(\abs{\xi}\abs{\phi})
  \end{equation*}
  with
  \begin{equation*}
    R_{\Phi_0} \coloneq \rho(\cdot)\Phi_0.
  \end{equation*}
  Since $\Phi_0$ is regular, $R_{\Phi_0}$ is injective, and it follows that $\xi = 0$ if $\abs{\phi} \lesssim \sigma \ll 1$ and $\bp$ is sufficiently close to $\bp_0$.
\end{proof}

\begin{proof}[Proof of \autoref{Prop_FueterKuranishiModel}]
  Denote by $\iota \co \coker (\rd\fF)_{s_0} \iso \coker \slD_\fH \to \Gamma(\fH)$ the inclusion of the $L^2$ orthogonal complement of $\im \slD_\fH$.
  Denote by $\pi_0 \co \Gamma(\fH) \to I_\del$ the $L^2$ orthogonal projection onto $I_\del = \ker (\rd\fF)_{s_0} \cap (\hat v\circ s)^\perp \subset \ker\slD_\fH \iso \ker(\rd\fF)_{s_0}$.
  Define
  \begin{equation*}
    \bar\slD_\fH \co \coker (\rd\fF)_{s_0} \oplus \Gamma(\fH) \to I_\del \oplus \R \oplus \Gamma(\fH)
  \end{equation*}
  by
  \begin{equation*}
    \bar\slD_\fH
    \coloneq
    \begin{pmatrix}
      0 & \pi_0 \\
      0 & -2\inner{\cdot}{\Phi_0}_{L^2} \\
      \iota & \slD_\fH
    \end{pmatrix}.
  \end{equation*}
  Set
  \begin{equation}
    \label{Eq_BarX0}
    \begin{split}
    \bar X_0
      &\coloneq \coker (\rd\fF)_{s_0} \oplus W^{k+1,p}\Gamma(\fH) \\
      &\quad\oplus W^{k+1,p}\Gamma(\fN) \\ 
      &\quad \oplus W^{k,p}\Omega^1(M,\fg_P) \oplus W^{k,p}\Omega^0(M,\fg_P) \qand \\
      \bar Y 
      &\coloneq I_\del \oplus \R \oplus W^{k,p}\Gamma(\fH) \\
      &\quad\oplus W^{k,p}\Gamma(\fN) \\
      &\quad\oplus W^{k+1,p}\Omega^1(M,\fg_P)\oplus W^{k+1,p}\Omega^0(M,\fg_P).
    \end{split}
  \end{equation}
  Define the operator $\bar L_{\bp,0} \co \bar X_0 \to \bar Y$ by
  \begin{equation}
    \label{Eq_LinearizationAtZeroBar}
    \bar L_{\bp,0}
    \coloneq
    \begin{pmatrix}
      -\bar\slD_\fH & \gamma\rII^* & 0 \\
      -\gamma\rII & -\slD_\fN & -\fa & 0 \\
      0 & -\fa^* & 0 
    \end{pmatrix}
  \end{equation}
  with $\fa \co \Omega^1(M,\fg_P) \oplus \Omega^0(M,\fg_P) \to \Gamma(\fN)$ defined by
  \begin{equation*}
    \fa(a,\xi) \coloneq \bar\gamma(a)\Phi_0 + \rho(\xi)\Phi.
  \end{equation*}

  The operator $\bar\slD_\fH$ is invertible because
  \begin{equation*}
    \begin{pmatrix}
      \pi_0 \\
      -2\inner{\cdot}{\Phi_0}_{L^2}
    \end{pmatrix}
  \end{equation*}
  is essentially the $L^2$ orthogonal projection onto $\ker \slD_\fH$.
  It can be verified by a direct computation that $\bar L_{\bp_0,0}$ is invertible and its inverse is given by
  \begin{equation}
    \label{Eq_L0Inverse}
    \begin{pmatrix}
      -\bar\slD_\fH^{-1} & 0 & -\bar\slD_\fH^{-1}\gamma\rII^*(\fa^*)^{-1} \\
      0 & 0 & -(\fa^*)^{-1} \\
      \fa^{-1}\gamma\rII \bar\slD_\fH^{-1} & -\fa^{-1} & \fa^{-1}\slD_\fN(\fa^*)^{-1} + \fa^{-1}\gamma\rII\bar\slD_\fH^{-1}\gamma\rII^*(\fa^*)^{-1} 
    \end{pmatrix}.
  \end{equation}
  After possibly shrinking $U$, we can assume that $\bar L_{\bp,0}$ is invertible for every $\bp \in U$.

  Since $Q_{\bp,0}$ is a quadratic map and
  \begin{equation}
    \label{Eq_Q0Estimate}
    \begin{split}
      \Abs{Q_{\bp,0}(\phi,a,\xi)}_Y
      &= \Abs{\bar\gamma(a)\phi}_{W^{k,p}}
        + \Abs{\mu(\phi)}_{W^{k+1,p}} 
        + \Abs{\phi}_{L^2}^2 \\
      &\lesssim
        \Abs{a}_{W^{k,p}}\Abs{\phi}_{W^{k+1,p}}
        + \Abs{\phi}_{W^{k+1,p}}^2,
    \end{split}
  \end{equation}
  $Q_{\bp,0}$ satisfies \eqref{Eq_QQuadraticEstimate};
  hence, we can apply \autoref{Lem_ZeroSetsOfFredholmMaps} to complete the proof.
\end{proof}

In the following regularity result, we decorate $X_0$ and $Y$ with superscripts indicating the choice of the differentiability and integrability parameters $k$ and $p$.

\begin{prop}
  \label{Prop_RegularityAtEpsilon=0}
  Assume the situation of \autoref{Prop_FueterKuranishiModel}.
  For each $k,\ell \in \N $ and $p,q \in (1,\infty)$ with $(k+1)p > 3$, $\ell \geq k$, and $q\geq p$, there are constants $c,\sigma > 0$ and an open neighborhood $U$ of $\bp_0$ in $\sP$ such that if $\bp \in U$ and $\hat \fc \in B_\sigma(0) \subset X^{k,p}_0$ is solution of
  \begin{equation*}
    L_{\bp,0}\hat\fc + Q_{\bp,0}(\hat\fc) + \fe_{\bp,0} = 0,
  \end{equation*}
  then $\hat \fc \in X^{\ell,q}_0$ and $\Abs{\hat\fc}_{X^{\ell,q}_0} \leq c \Abs{\hat\fc}_{X^{k,p}_0}$.
\end{prop}

\begin{proof}
  Provided $U$ is a sufficiently small neighborhood of $\bp_0$ and $0 < \sigma \ll 1$, it follows from Banach's Fixed Point Theorem that $(0,\hat \fc)$ is the unique solution in $B_\sigma(0)\subset \bar X^{k,p}$ of
  \begin{equation*}
    \bar L_{\bp,0}(0,\hat\fc) + Q_{\bp,0}(\hat\fc) + \fe_{\bp,0}
    =
    \begin{pmatrix}
      \pi\hat\fc \\ 0
    \end{pmatrix},
  \end{equation*}
  and that there exists a $(o,\hat\fd) \in B_\sigma(0)\subset \bar X^{\ell,q}$ such that
  \begin{equation*}
    \bar L_{\bp,0}(o,\hat\fd) + Q_{\bp,0}(\hat\fd) + \fe_{\bp,0}
    =
    \begin{pmatrix}
      \pi\hat\fc \\ 0
    \end{pmatrix}.
  \end{equation*}
  Since $\bar X^{\ell,q} \subset \bar X^{k,p}$ and $\Abs{(o,\hat\fd)}_{\bar X^{k,p}} \leq \Abs{(o,\hat\fd)}_{\bar X^{\ell,q}} \leq \sigma$, it follows that $(o,\hat\fd) = (0,\hat\fc)$ and thus $\hat\fc \in \bar X^{\ell,q}$ and $\Abs{\hat\fc}_{X^{\ell,q}} \leq \sigma$.
  From this it follows easily that $\Abs{\hat\fc}_{X^{\ell,q}} \leq c\Abs{\hat\fc}_{X^{k,p}}$.
\end{proof}


\section{Deformation theory around \texorpdfstring{$\epsilon = 0$}{epsilon = 0}}
\label{Sec_DeformationTheoryNearZero}

In this section we will prove \autoref{Thm_KuranishiModelNearZero}, whose hypotheses we will assume throughout.

Fix $k \in \N$ and $p \in (1,\infty)$ with $(k+1)p > 3$.
Let
\begin{multline*}
  \fM_\SW^{k,p}
  =
  \Bigg\{
    (\bp,\epsilon,[(\Phi,A)]) \in \sP\times \R^+ \times \frac{W^{k+1,p}\Gamma(\fS)\times W^{k+2,p}\sA(P)}{{W^{k+3,p}\sG(P)}}
    :
    (\epsilon,\Phi,A) \text{ satisfies }
    \eqref{Eq_BlownUpSeibergWitten}
  \Bigg\}.
\end{multline*}
For $\ell \in \N$ and $q \in (1,\infty)$ with $\ell \geq k$ and $q \geq p$, the inclusions $\fM_\SW^{\ell,q} \subset \fM_\SW^{k,p} \subset \fM_\SW$ are homeomorphisms;
see also \autoref{Prop_RegularityNearEpsilon=0}.

\subsection{Reduction to a slice}

\begin{prop}
  \label{Prop_GaugeFixingNearZero}
  Let $\fc_0 = (\Phi_0,A_0) \in \Gamma(\fS^\reg)\times \sA(P)$ and $\bp_0 \in \sP$.
  For $\bp \in \sP$, set
  \begin{equation*}
    X_\epsilon
    \coloneq W^{k+1,p}\Gamma(\fS) \oplus W^{k+2,p}\Omega^1(M,\fg_P)\oplus W^{k+2,p}\Omega^0(M,\fg_P)
  \end{equation*}
  and
  \begin{equation*}
    \Abs{(\phi,a,\xi)}_{X_\epsilon}
    \coloneq \Abs{\phi}_{W^{k+1,p}}
    + \Abs{(a,\xi)}_{W^{k,p}}
    + \epsilon \Abs{\nabla^{k+1}(a,\xi)}_{L^p}
    + \epsilon^2 \Abs{\nabla^{k+2}(a,\xi)}_{L^p}.
  \end{equation*}
  There exist a neighborhood $U$ of $\bp_0 \in \sP$ and constants $\sigma,\epsilon_0,c > 0$ such that the following holds.
  If $\bp \in U$, $\hat\fc = (\phi,a) \in X_\epsilon$, and $\epsilon \in (0,\epsilon_0]$ are such that
  \begin{equation*}
    \Abs{\hat\fc}_{X_\epsilon} < \sigma,
  \end{equation*}
  then there exists a $W^{k+3,p}$ gauge transformation $g$ such that $(\tilde\phi,\tilde a) = g(\fc_0+\hat\fc) - \fc_0$ satisfies
  \begin{equation*}
    \Abs{(\tilde\phi,\tilde a)}_{X_\epsilon} < c\sigma,
  \end{equation*}
  and
  \begin{equation}
    \label{Eq_GaugeFixingNearZero}
    \epsilon^2 \rd_{A_0B}^*\tilde a - \rho^*(\tilde\phi\Phi_0^*) = 0.
  \end{equation}
\end{prop}

\begin{proof}
  To construct $g$, note that for $g = e^\xi$ with $\xi \in W^{k+3,p}\Omega^0(M,\fg_P)$ we have
  \begin{equation*}
    \tilde \phi = \rho(\xi)\Phi_0 + \rho(\xi)\phi + \fm(\xi)
    \qandq
    \tilde a = a - \rd_{A_0}\xi - [a,\xi] + \fn(\xi).
  \end{equation*}
  Here $\fn$ and $\fm$ denote expressions which are algebraic and at least quadratic in $\xi$.
  The gauge fixing condition \eqref{Eq_GaugeFixingNearZero} can thus be written as
  \begin{equation*}
    \fl_\epsilon \xi + \fd_\epsilon\xi + \fq_\epsilon(\xi) + \fe_\epsilon = 0.
  \end{equation*}
  with
  \begin{align*}
    \fl_\epsilon
    &\coloneq
      \epsilon^2 \Delta_{A_0B}  + R_{\Phi_0}^*R_{\Phi_0}, &
    \fd_\epsilon
    &\coloneq
      \epsilon^2 \rd_{A_0B}^*[a,\cdot] + \rho^*(\rho(\cdot)\phi\Phi_0^*), \\
    \fq_\epsilon(\xi)
    &\coloneq
      \epsilon^2 \rd_{A_0B}^*\fn(\xi) + \rho^*(\fm(\xi)\Phi_0^*), &
    \fe_\epsilon
    &\coloneq
      - \epsilon^2\rd^*_{A_0B}a - \rho^*(\phi\Phi_0).
  \end{align*}

  Denote by $G_\epsilon$ the Banach space $W^{k+3,p}\Omega^0(M,\fg_P)$ equipped with the norm
  \begin{equation}
    \label{Eq_GaugeFixingNearZeroPDE}
    \Abs{\xi}_{G_\epsilon}
    \coloneq \Abs{\xi}_{W^{k+1}}
    + \epsilon\Abs{\nabla^{k+2}\xi}_{L^p}
    + \epsilon^2\Abs{\nabla^{k+3}\xi}_{L^p}.
  \end{equation}
  Since $\Phi_0$ is regular, the operator $R_{\Phi_0}^*R_{\Phi_0}$ is positive definite;
  hence, for $\epsilon \ll 1$, the operator
  \begin{equation*}
    \fl_\epsilon \co G_\epsilon \to W^{k+1,p}\Omega^0(M,\fg_P)
  \end{equation*}
  is invertible and $\Abs{\fl_\epsilon^{-1}}_{\cL(G_\epsilon,W^{k+1,p})}$ is bounded independent of $\epsilon$. 
  Since
  \begin{equation*}
    \Abs{\fd_\epsilon}_{\cL(G_\epsilon,W^{k+1,p})} \lesssim \sigma \ll 1,
  \end{equation*}
  $\fl_\epsilon + \fd_\epsilon\co G_\epsilon \to W^{k+1,p}\Omega^0(M,\fg_P)$ will also be invertible with inverse bounded independent of $\epsilon$ and $\sigma$.
  Since the non-linearity $\fq_\epsilon \co G_\epsilon \to W^{k+1,p}\Omega^0(M,\fg_P)$ satisfies \eqref{Eq_QQuadraticEstimate} and $\Abs{\fe_\epsilon} \lesssim \sigma \ll 1$, it follows from Banach's Fixed Point Theorem that, for a suitable $c> 0$, there exists a unique solution $\xi \in B_{c\sigma}(0) \subset G_\epsilon$ to \eqref{Eq_GaugeFixingNearZeroPDE}.
  This proves the existence of the desired gauge transformation, and local uniqueness.
  Global uniqueness follows by an argument by contradiction, cf.~\cite[Proposition 4.2.9]{Donaldson1990}.
\end{proof}

\begin{prop}
  \label{Prop_Lepsilon}
  Let $\fc_0 = (\Phi_0,A_0)$ be a lift of a Fueter section $s_0 \in \Gamma(\fX)$ for $\bp_0 \in \sP$.
  Fix $\epsilon > 0$ and $\bp \in \sP$.
  Define a linear map $L_{\bp,\epsilon} \co X_\epsilon \to Y$ and a quadratic map $Q_{\bp,\epsilon} \co X_0 \to Y$ by
  \begin{align*}
    L_{\bp,\epsilon}
    &\coloneq
    \begin{pmatrix}
      -\slD_{A_0} & -\gamma(\cdot)\Phi_0 & -\rho(\cdot)\Phi_0 \\
      -2*\mu(\Phi_0,\cdot) & *\epsilon^2\rd_{A_0} & \epsilon^2\rd_{A_0} \\
      -\rho^*(\cdot\,\Phi_0^*) & \epsilon^2\rd_{A_0}^* & 0 \\
      2\inner{\Phi_0}{\cdot}_{L^2} & 0 & 0
    \end{pmatrix}
    \qand \\
    Q_{\bp,\epsilon}(\phi,a,\xi)
    &\coloneq
    \begin{pmatrix}
      -\bar\gamma(a)\phi \\
      \frac12\epsilon^2*[a\wedge a] -*\mu(\phi) \\
      0 \\
      \Abs{\phi}_{L^2}^2
    \end{pmatrix},
  \end{align*}
  respectively.
  With $\fe_{\bp,0}$ as in \autoref{Prop_L0} set
  \begin{equation*}
    \fe_{\bp,\epsilon} \coloneq \fe_{\bp,0} + \epsilon^2(0,*\varpi F_{A_0},0).
  \end{equation*}
  There exist a neighborhood $U$ of $\bp_0 \in \sP$ and $\sigma > 0$ such that $\hat\fc = (\phi,a,\xi) \in B_\sigma(0) \subset X_\epsilon$ satisfies
  \begin{equation}
    \label{Eq_Lepsilon}
    L_{p,\epsilon} \hat\fc + Q_{\bp,\epsilon}(\hat\fc) + \fe_{\bp,\epsilon} = 0
  \end{equation}
  if and only if $\xi = 0$, $(A,\Phi) = (A_0+a,\Phi_0+\phi)$ satisfies
  \begin{equation*}
    \slD_A \Phi = 0, \quad
    \epsilon^2\varpi F_A = \mu(\Phi),
    \qandq \Abs{\Phi}_{L^2} = 1,
  \end{equation*}
  and
  \begin{equation}
    \label{Eq_GaugeFixingNearZero_0}
    \epsilon^2 \rd_{A_0}^*a - \rho^*(\phi\Phi_0^*) = 0.
  \end{equation}
\end{prop}

\begin{proof}
  We only need to show that $\xi$ vanishes, but this follows from the same argument as in the proof of \autoref{Prop_L0} because $\rd_AF_A = 0$.
\end{proof}

\begin{cor}
  \label{Cor_HomeomorphismTofMSWEpsilon}
  There exist $\epsilon,\sigma>0$ such the map
  \begin{equation*}
    \set*{
      (\bp,\epsilon,\phi,a,\xi) \in \sP \times U \times (0,\epsilon_0) \times B_\sigma(0)
      \text{ satisfying }
      \eqref{Eq_Lepsilon}
  } \to \fM_\SW
  \end{equation*}
  defined by
  \begin{equation*}
    (\bp,\epsilon,\phi,a,\xi) \mapsto (\bp,\epsilon,[(\Phi_0+\phi,A_0+a)])
  \end{equation*}
  is a homeomorphism onto the intersection of $\fM_\SW$ with a neighborhood of $([\fc_0],\bp_0,0)$ in $\overline{\fM}_\SW$.
\end{cor}

\subsection{Inverting \texorpdfstring{$\bar L_{\bp,\epsilon}$}{the completed linearization}}

Define the Banach space $(\bar X_\epsilon,\Abs{\cdot}_{\bar X_\epsilon})$ by
\begin{equation*}
  \bar X_\epsilon
  \coloneq
    \coker (\rd\fF)_{s_0} \oplus W^{k+1,p}\Gamma(\fS)
    \oplus W^{k+2,p}\Omega^1(M,\fg_P)\oplus W^{k+2,p}\Omega^0(M,\fg_P)
\end{equation*}
with norm
\begin{equation*}
  \Abs{(o,\hat\fc)}_{\bar X_\epsilon}
  \coloneq
    \abs{o} + \Abs{\hat\fc}_{X_\epsilon},
\end{equation*}
and the Banach space $(\bar Y,\Abs{\cdot}_{\bar Y})$ by
\begin{equation*}
  \bar Y
  \coloneq I_\del \oplus \R \oplus W^{k,p}\Gamma(\fS)
  \oplus W^{k+1,p}\Omega^1(M,\fg_P)\oplus W^{k+1,p}\Omega^0(M,\fg_P)
\end{equation*}
with the obvious norm.
Let $\bar\slD_\fH \co \coker(\rd\fF)_{s_0} \oplus W^{k+1,p}\Gamma(\fS) \to I_\del \oplus \R \oplus W^{k,p}\Gamma(\fS)$ be as in the Proof of \autoref{Prop_FueterKuranishiModel}.
Define $\bar L_{\bp,\epsilon}\co \bar X_\epsilon \to \bar Y$ by
\begin{equation}
  \bar L_{\bp,\epsilon}
  \coloneq
    \begin{pmatrix}
      -\bar\slD_\fH & \gamma\rII^* &  0 \\
      -\gamma\rII & -\slD_\fN & -\fa \\
      0 & -\fa^* & \epsilon^2\delta_{A_0} 
    \end{pmatrix}
\end{equation}
with
\begin{equation*}
  \delta_{A_0}
  \coloneq
  \begin{pmatrix}
    *\rd_{A_0} & \rd_{A_0} \\
    \rd_{A_0}^* & 0
  \end{pmatrix}.
\end{equation*}

\begin{prop}
  \label{Prop_SmallEpsilonLinearisationIsInvertible}
  There exist $\epsilon_0,c > 0$, and a neighborhood $U$ of $\bp_0 \in \sP$ such that, for all  $\bp \in U$ and $\epsilon \in (0,\epsilon_0]$, $\bar L_{\bp,\epsilon} \co \bar X_\epsilon \to \bar Y$ is invertible, and $\Abs*{\bar L_{\bp,\epsilon}^{-1}}
  \leq c$.
\end{prop}

The proof of this result relies on the following two observations.

\begin{prop}
  \label{Prop_InvertibleMatrix}
  For $i = 1, 2, 3$, let $V_i$ and $W_i$ be Banach spaces, and set
  \begin{equation*}
      V \coloneq \bigoplus_{i=1}^3 V_i \qandq
      W \coloneq \bigoplus_{i=1}^3 W_i.
  \end{equation*}  
  Let $L \co V \to W$ be a bounded linear operator of the form
  \begin{equation*}
    L = 
    \begin{pmatrix}
      D_1 & B_+ & 0 \\
      B_- & D_2 & A_+ \\
      0 & A_- & D_3
    \end{pmatrix}.
  \end{equation*}
  If the operators
  \begin{align*}
  &D_1\co V_1 \to W_1, \\
  &A_- \co V_2 \to W_3, \qandq \\
  &Z \coloneq A_+ - (D_2 - B_- D_1^{-1} B_+) A_-^{-1} D_3 \co V_3 \to W_2
  \end{align*}
   are invertible, then there exists a bounded linear operator $R \co W \to V$ such that
  \begin{equation*}
    RL = \id_W.
  \end{equation*}
  Moreover, the operator norm $\Abs{R}$ is bounded by a constant depending only on $\Abs{L}$, $\Abs{D_1^{-1}}$, $\Abs{A_-^{-1}}$, and $\Abs{Z^{-1}}$. 
\end{prop}

\begin{prop}
  \label{Prop_ZInvertible}
  There exist $\epsilon_0,c > 0$ such that for $\epsilon \in (0,\epsilon_0]$, the linear map
  \begin{multline*}
    \fz_\epsilon \coloneq \fa + \epsilon^2\(\slD_\fN  + \gamma \rII \slD_\fH^{-1} \gamma \rII^*\) (\fa^*)^{-1}\delta_{A_0} \co W^{k+2,p}\Omega^1(M,\fg_P)\oplus W^{k+2,p}\Omega^0(M,\fg_P)
    \to W^{k,p}\Gamma(\fN)
  \end{multline*}
  is invertible, and
  \begin{equation*}
    \Abs{\fz_\epsilon^{-1}(a,\xi)}_{W^{k,p}}
    + \epsilon\Abs{\nabla^{k+1}\fz_\epsilon^{-1}(a,\xi)}_{L^p}
    + \epsilon^2\Abs{\nabla^{k+2}\fz_\epsilon^{-1}(a,\xi)}_{L^p}
    \leq c \Abs{(a,\xi)}_{W^{k,p}}.
  \end{equation*}
\end{prop}

\begin{proof}[Proof of \autoref{Prop_SmallEpsilonLinearisationIsInvertible}]
  It suffices to prove the result for $\bp = \bp_0$, for then it follows for $\bp$ close to $\bp_0$.

  Recall that
  \begin{align*}
    \bar X_\epsilon
    &= \coker (\rd\fF)_{s_0} \oplus W^{k+1,p}\Gamma(\fH) \\
    &\quad\oplus W^{k+1,p}\Gamma(\fN) \\
    &\quad\oplus W^{k+2,p}\Omega^1(M,\fg_P)\oplus W^{k+2,p}\Omega^0(M,\fg_P), \\
    \bar Y
    &= I_\del \oplus \R \oplus W^{k,p}\Gamma(\fH) \\
    &\quad\oplus W^{k,p}\Gamma(\fN) \\
    &\quad\oplus W^{k+1,p}\Omega^1(M,\fg_P)\oplus W^{k+1,p}\Omega^0(M,\fg_P),
  \end{align*}
  and $\bar L_{\bp_0,\epsilon}$ can be written as
  \begin{equation*}
    \begin{pmatrix}
      -\bar\slD_\fH & \gamma\rII^* & 0 \\
      -\gamma\rII & -\slD_\fN & -\fa \\
      0 & -\fa^* & \epsilon^2\delta_{A_0}
    \end{pmatrix}  
  \end{equation*}
  with
  \begin{equation*}
    \delta_{A_0}
    =
    \begin{pmatrix}
      *\rd_{A_0} & \rd_{A_0} \\
      \rd_{A_0}^* & 0
    \end{pmatrix}.
  \end{equation*}

  The operators $\bar\slD_\fH \co \coker (\rd\fF)_{s_0} \oplus W^{k+1,p}\Gamma(\fH) \to I_\del \oplus \R \oplus W^{k,p}\Gamma(\fH)$ and $\fa^*\co W^{k+1,p}\Gamma(\fN) \to W^{k+1,p}\Omega^1(M,\fg_P) \oplus W^{k+1,p}\Omega^0(M,\fg_P)$ both are invertible with uniformly bounded inverses, and by \autoref{Prop_ZInvertible} the same holds for $\fz_\epsilon$, provided $\epsilon \ll 1$.
  Thus, according to \autoref{Prop_InvertibleMatrix}, $\bar L_{\bp_0,\epsilon}$ has a left inverse $R_\epsilon \co \bar Y_0 \to \bar X_\epsilon$ whose norm can be bounded independent of $\epsilon$.

  To see that $R_\epsilon$ is also a right inverse, observe that $L_{\bp_0,\epsilon}$ is a formally self-adjoint elliptic operator and, hence, $L_{\bp_0,\epsilon} \co X_\epsilon \to Y$ is Fredholm of index zero.
  Consequently, $\bar L_{\bp_0,\epsilon}$ is Fredholm of index zero.
  The existence of $R_\epsilon$ shows that $\ker \bar L_{\bp_0,\epsilon} = 0$ and thus $\coker \bar L_{\bp_0,\epsilon} = 0$.
  By the Open Mapping Theorem, $\bar L_{\bp_0,\epsilon}$ has an inverse $\bar L_{\bp_0,\epsilon}^{-1}$.
  It must agree with $R_\epsilon$ since $R_\epsilon = R_\epsilon \bar L_{\bp_0,\epsilon} \bar L_{\bp_0,\epsilon}^{-1} = \bar L_{\bp_0,\epsilon}^{-1}$.
\end{proof}

\begin{proof}[Proof of \autoref{Prop_InvertibleMatrix}]
  The left inverse of $L$ can be found by Gauss elimination \cite[Chapter 2]{Strang2016}.
  The formula found in this way is rather unwieldy;
  fortunately, however, the precise formula is not needed.
  \setcounter{step}{0}
  \begin{step}
    Set
    \begin{equation*}
      E \coloneq (D_2 -B_- D_1^{-1} B_+ ) A_-^{-1} \co W_3 \to W_2
  \end{equation*}
    The linear map $P \co  W \to V$ defined by
    \begin{equation*}
      P \coloneq \begin{pmatrix}
        D_1^{-1} & 0 & 0 \\
        0 & 0 & A_-^{-1} \\
        - Z^{-1} B_- D_1^{-1} & Z^{-1} & - Z^{-1} E
      \end{pmatrix}
    \end{equation*}
    satisfies
    \begin{equation*}
      PL = 
      \begin{pmatrix}
        \id_{V_1} & D_1^{-1} B_+ & 0 \\
        0 & \id_{V_2} & A_-^{-1} D_3 \\
        0 & 0  & \id_{ V_3}.
      \end{pmatrix}.
    \end{equation*}    
    Moreover, $\Abs{P}$ and $\Abs{PL}$ are bounded by a constant depending only $\Abs{L}$, $\Abs{D_1^{-1}}$, $\Abs{A_-^{-1}}$, and $\Abs{Z^{-1}}$. 
  \end{step}

  This can be verified directly;
  alternatively, one can check that a sequence of row operations transforms the augmented matrix $( L \,|\, \id)$ as follows: 
  \begingroup
  \allowdisplaybreaks
  \begin{align*}
    &
      \begin{pmatrix}[ccc|ccc]
        D_1 & B_+ & 0 & \id_{W_1} & 0 & 0 \\
        B_- & D_2 & A_+ & 0 & \id_{W_2} & 0 \\
        0 & A_- & D_3 & 0 & 0 & \id_{W_3}
      \end{pmatrix} \\
    \leadsto
    &
      \begin{pmatrix}[ccc|ccc]
        \id_{V_1} & D_1^{-1} B_+ & 0 & D_1^{-1} & 0 & 0 \\
        B_- & D_2 & A_+ & 0 & \id_{W_2} & 0 \\
        0 & \id_{V_2} & A_{-}^{-1} D_3 & 0 & 0 & A_-^{-1}
      \end{pmatrix} \\
    \leadsto
    &
      \begin{pmatrix}[ccc|ccc]
        \id_{V_1} & D_1^{-1} B_+ & 0 & D_1^{-1} & 0 & 0 \\
        0 & \id_{V_2} & A_{-}^{-1} D_3 & 0 & 0 & A_-^{-1} \\
        B_- & D_2 & A_+ & 0 & \id_{W_2} & 0 
      \end{pmatrix}\\
    \leadsto 
    &
      \begin{pmatrix}[ccc|ccc]
        \id_{V_1} & D_1^{-1} B_+ & 0 & D_1^{-1} & 0 & 0 \\
        0 & \id_{V_2} & A_{-}^{-1} D_3 & 0 & 0 & A_-^{-1} \\
        0 & D_2 - B_- D_1^{-1} B_+ & A_+ & - B_- D_1^{-1} & \id_{W_2} & 0 
      \end{pmatrix} \\
    \leadsto
    & 
      \begin{pmatrix}[ccc|ccc]
        \id_{V_1} & D_1^{-1} B_+ & 0 & D_1^{-1} & 0 & 0 \\
        0 & \id_{V_2} & A_{-}^{-1} D_3 & 0 & 0 & A_-^{-1} \\
        0 & 0 & Z & - B_- D_1^{-1} & \id_{W_2} & - E
      \end{pmatrix}\\
    \leadsto
    & 
      \begin{pmatrix}[ccc|ccc]
        \id_{V_1} & D_1^{-1} B_+ & 0 & D_1^{-1} & 0 & 0 \\
        0 & \id_{V_2} & A_{-}^{-1} D_3 & 0 & 0 & A_-^{-1} \\
        0 & 0 & \id_{V_3} & - Z^{-1} B_- D_1^{-1} & Z^{-1} & - Z^{-1}E
      \end{pmatrix}.
  \end{align*}
  \endgroup

  \begin{step}
    The inverse of $PL$ is
    \begin{equation*}
      (PL)^{-1} =
      \begin{pmatrix}
        \id_{V_1} & -D_1^{-1} B_+ & D_1^{-1} B_+A_-^{-1} D_3 \\
        0 & \id_{V_2} & -A_-^{-1} D_3 \\
        0 & 0  & \id_{ V_3}.
      \end{pmatrix}.
    \end{equation*}
    Hence, $R \coloneq (PL)^{-1}P$ is the desired left inverse.
  \end{step}

  It can be verified directly that the above expression gives the inverse of $PL$.
\end{proof}

\begin{proof}[Proof of \autoref{Prop_ZInvertible}]
  It suffices to show that the linear maps $\tilde \fz_\epsilon \coloneq \fa^*\fz_\epsilon 
  $ are uniformly invertible.
  A short computation using \autoref{Prop_DMu_D*Mu} shows that
  \begin{equation*}
    \tilde \fz_\epsilon
    = \epsilon^2\delta_{A_0}^2 + \fa^*\fa + \epsilon^2\fe
  \end{equation*}
  where $\fe$ is a zeroth order operator which factors through $W^{k+1,p} \to W^{k+1,p}$.
  Since $\Phi_0$ is regular, $\fa^*\fa$ is positive definite and, hence, for $\epsilon \ll 1$, $\fa^*\fa + \epsilon^2\delta_{A_0}^2$ is uniformly invertible.
  Since $\epsilon \ll 1$, $\epsilon^2\fe$ is a small perturbation of order $\epsilon$ and thus $\tilde\fz_\epsilon$ is uniformly invertible.
\end{proof}

The above analysis yields the following regularity result, in which we decorate $X_\epsilon$ and $Y$ with superscripts indicating the choice of the differentiability and integrability parameters $k$ and $p$.
The proof is almost identical to that of \autoref{Prop_RegularityAtEpsilon=0}, and will be omitted.

\begin{prop}
  \label{Prop_RegularityNearEpsilon=0}
  For each $k,\ell \in \N $ and $p,q \in (1,\infty)$ with $(k+1)p > 3$, $\ell \geq k$, and $q\geq p$, there are constants $c,\sigma,\epsilon_0 > 0$ and an open neighborhood $U$ of $\bp_0$ in $\sP$ such that if $\epsilon \in (0,\epsilon_0]$, $\bp \in U$, and $\hat \fc \in B_\sigma(0) \subset X^{k,p}_\epsilon$ is solution of
  \begin{equation*}
    L_{\bp,\epsilon}\hat\fc + Q_{\bp,\epsilon}(\hat\fc) + \fe_{\bp,\epsilon} = 0,
  \end{equation*}
  then $\hat \fc \in X^{\ell,q}_\epsilon$ and $\Abs{\hat\fc}_{X^{\ell,q}_\epsilon} \leq c \Abs{\hat\fc}_{X^{k,p}_\epsilon}$.
\end{prop}

\subsection{Proof of \autoref{Thm_KuranishiModelNearZero}}

Since $Q_{\bp,\epsilon}$ is quadratic and
\begin{align*}
  \Abs{Q_{\bp,\epsilon}(\phi,a,\xi)}_Y
  &\leq
    \Abs{\bar\gamma(a)\phi}_{W^{k,p}}
      + \epsilon^2 \Abs{[a\wedge a]}_{W^{k+1,p}}
      + \Abs{\mu(\phi)}_{W^{k+1,p}} 
      + \Abs{\phi}_{L^2}^2 \\
  &\lesssim
      \Abs{a}_{W^{k,p}}\Abs{\phi}_{W^{k+1,p}} \\
    &\quad
      + \(\Abs{a}_{W^{k,p}} + \epsilon \Abs{\nabla^{k+1}a}_{L^p} + \epsilon^2 \Abs{\nabla^{k+2} a}_{L^p}\)^2
      + \Abs{\phi}_{W^{k+1,p}}^2,
\end{align*}
$Q_{\bp,\epsilon}$ satisfies \eqref{Eq_QQuadraticEstimate},
and because of \autoref{Prop_SmallEpsilonLinearisationIsInvertible}
we can apply \autoref{Lem_ZeroSetsOfFredholmMaps} to construct a smooth map $\ob_{\circ} \co U\times (0,\epsilon_0)\times\sI_\del \to \coker (\rd\fF)_{s_0}$ and a map $\fx_{\circ} \co \ob^{-1}(0) \to \overline\fM_\SW$ which is a homeomorphism onto the intersection of $\fM_\SW$ with a neighborhood of $[(A_0,\Phi_0)]$.
(There is a slight caveat in the application of \autoref{Lem_ZeroSetsOfFredholmMaps}: the Banach space $X_\epsilon$ does depend on $\bp$ and $\epsilon$ and $Y$ depends on $\bp$.
The dependence, however, is mostly harmless as different values of $\bp$ and $\epsilon$ lead to naturally isomorphic Banach spaces.)
For what follows it will be important to know that maps $\ob_{\circ}$ and $\fx_{\circ}$ are uniquely characterized as follows:
for $\bp$ in the open neighborhood $U$ of $\bp_0 \in \sP$,
$d$ in the open neighborhood $\sI_\del$ of $0 \in I_\del$,
and $\epsilon \in (0,\epsilon_0)$,
there is a unique solution $\bar \fc = \bar \fc(\bp,\epsilon,d) \in B_\sigma(0) \subset \bar X_\epsilon$ of
\begin{equation}
  \label{Eq_ExtendedSW}
  \bar L_{\bp,\epsilon}\bar\fc + Q_{\bp,\epsilon}(\bar\fc) + \fe_{\bp,\epsilon} = d \in \sI_\del \subset \bar Y;
\end{equation}
$\ob_{\circ}(\bp,\epsilon,d)$ is the component of $\bar \fc(\bp,\epsilon,d)$ in $\coker (\rd\fF)_{s_0}$
and if $\ob_{\circ}(\bp,\epsilon,d) = 0$ and $\hat\fc$ denotes the component of $\fc(\bp,\epsilon,d)$ in $X_\epsilon$,
then $\fx_{\circ}(\bp,\epsilon,d) = \fc_0 + \hat\fc$.
(Similar, setting $\epsilon = 0$ yields $\ob_\del$ and $\fx_\del$.)

We define $\ob \co U \times [0,\epsilon_0) \times \sI_\del \to \coker (\rd\fF)_{s_0}$ by
\begin{equation*}
  \ob(\cdot,\epsilon,\cdot) =
  \begin{cases}
    \ob_{\circ}(\cdot,\epsilon,\cdot) & \text{for } \epsilon \in (0,\epsilon_0) \\
    \ob_{\del}(\cdot,\cdot) & \text{for } \epsilon = 0,
  \end{cases}
\end{equation*}
and $\fx \co \ob^{-1}(0) \to \overline\fM_\SW$ by
\begin{equation*}
    \fx(\cdot,\epsilon,\cdot) =
  \begin{cases}
    \fx_{\circ}(\cdot,\epsilon,\cdot) & \text{for } \epsilon \in (0,\epsilon_0) \\
    \fx_{\del}(\cdot,\cdot) & \text{for } \epsilon = 0.
  \end{cases}
\end{equation*}
In order to prove \autoref{Thm_KuranishiModelNearZero} we need to understand the regularity of $\ob$ near $\epsilon = 0$;
in other words: we need to understand how $\ob_{\circ}$ and $\ob_\del$ fit together.

Let $k \in \N$ and $p \in (1,\infty)$ be the differentiability and integrability parameters used in the definition of $\bar X_\epsilon$.
If necessary, shrink $U$ and $\sI_\del$ and decrease $\sigma$ so that the proof of \autoref{Prop_FueterKuranishiModel} goes through and \autoref{Prop_L0} holds with differentiability parameter $k+2r+2$ and integrability parameter $p$.
Observe that $\bar X_0^{k+2,p} \subset \bar X_\epsilon$ and the norm of the inclusion can be bounded by a constant independent of $\epsilon$.

\begin{prop}
  \label{Prop_Approximation}
  For every $(\bp,d) \in U\times\sI_\del$,
  there are $\bar\fc_0(\bp,d) \in \bar X_0^{k+2r+2,p}$ and
  $\hat\fc_i(\bp,d) \in \bar X_0^{k+2(r-i)+2,p}$ $($for $i = 1, \ldots, r${}$)$ depending smoothly on $\bp$ and $d$,
  such that, for $m,n \in \N$ with $m+n \leq 2r$,
  \begin{equation*}
    \tilde\fc(\bp,\epsilon,d)
    \coloneq
      \bar\fc_0 + \sum_{i=1}^r \epsilon^{2i}\hat\fc_i
  \end{equation*}
  satisfies
  \begin{equation}
    \label{Eq_ApproximationEstimate}
    \Abs*{\nabla^m_{U\times\sI_\del}\del_\epsilon^n\paren*{
      \bar\fc(\bp,\epsilon,d) -  \tilde\fc(\bp,\epsilon,d)
    }}_{\bar X_\epsilon} = O(\epsilon^{2k+2-n}).
  \end{equation}
\end{prop}

\begin{proof}
  We construct $\tilde\fc$ by expanding \eqref{Eq_ExtendedSW} in $\epsilon^2$.
  To this end we write
  \begin{equation*}
    \bar L_{\bp,\epsilon} = \bar L_{\bp,0} + \epsilon^2 \ell_\bp, \quad
    Q_{\bp,\epsilon}, = Q_{\bp,0} + \epsilon^2 q_\bp, \qandq
    \fe_{\bp,\epsilon} = \fe_{\bp,0} + \epsilon^2 \hat\fe_\bp,
  \end{equation*}
  with
  \begin{equation*}
    \ell_\bp 
    \coloneq
    \begin{pmatrix}
      0 & &  \\
      & 0 & \\
      & & \delta_{A_0}
    \end{pmatrix}, \quad
    q_{\bp}(\phi,a,\xi)
    \coloneq
    \begin{pmatrix}
      0 \\
      0 \\
      \frac12 *[a\wedge a]
    \end{pmatrix}, \qandq
    \hat\fe_\bp
    \coloneq
    \begin{pmatrix}
      0 \\
      0 \\
      *\varpi F_{A_0}
    \end{pmatrix}.
  \end{equation*}
  Observe that $\ell_\bp \co \bar X^{\ell,p}_0 \to \bar Y^{\ell-2,p}$ is a bounded linear map and $q_\bp\co \bar X^{\ell,p}_0 \to \bar Y^{\ell-2,p}$ is a bounded quadratic map.

  \setcounter{step}{0}
  \begin{step}
    Construction of $\bar\fc_0$ and $\hat\fc_i$.
  \end{step}

  By Banach's Fixed Point Theorem, there is a unique solution $\bar\fc_0 \in B_{\sigma}(0) \subset \bar X_0^{k+2r+2,p}$ of
  \begin{equation*}
    \bar L_{\bp,0}\bar\fc_0 + Q_{\bp,0}(\bar\fc_0) + \fe_{\bp,0} = d \in \sI_\del \subset \bar Y^{k+2r+2}.
  \end{equation*}
  Moreover, $\bar\fc_0$ actually lies in $B_{\sigma/2}(0) \subset \bar X_0^{k+2r+2,p}$ provided $U$ and $\sI_\del$ have been chosen sufficiently small.
  We have
  \begin{equation*}
    \bar L_{\bp,\epsilon}\bar\fc_0 + Q_{\bp,0}(\bar\fc_\epsilon) + \fe_{\bp,\epsilon} - d
    =
     \epsilon^2 \fr_0(\bp,d) \in \bar Y^{k+2(r-1)+2,p}.
  \end{equation*}
  with
  \begin{equation*}
    \fr_0(\bp,d) \coloneq \ell_\bp\bar\fc_0 + q_{\bp}(\bar\fc_0) + \hat\fe_\bp.
  \end{equation*}
  Since $\sigma \ll 1$, the operator $\bar L_{\bp,0}+ 2Q_{\bp,0}(\bar\fc_0,\cdot) \co \bar X_0^{k+2(r-i)+2,p} \to \bar Y_0^{k+2(r-i)+2,p}$ is invertible for $i=1,\ldots,r$.%
  \footnote{%
    Here we engage in the slight abuse of notation to use the same notation for a bilinear map and its associated quadratic form.
  }
  Recursively define $\fr_i(\bp,d) \in \bar Y^{k+2(r-i-1)+2,p}$ by
  \begin{equation*}
    \epsilon^{2i+2}\fr_i
    \coloneq
      \bar L_{\bp,\epsilon}\bar\fc_\epsilon^i + Q_{\bp,0}(\bar\fc_\epsilon^i) + \fe_{\bp,\epsilon} - d
  \end{equation*}
  with
  \begin{equation*}    
    \tilde\fc(\epsilon,\bp,d)
    \coloneq
      \bar\fc_0 + \epsilon^2\hat\fc_1 + \cdots + \epsilon^{2i}\hat\fc_i,
  \end{equation*}
  and define $\hat\fc_{i+1} \in \bar X_0^{k+2(r-i-1)+2}$ to be the unique solution of
  \begin{equation*}
    \bar L_{\bp,0}\hat\fc_{i+1} + 2Q_{\bp,0}(\bar\fc_0,\hat\fc_{i+1})
    =
    \fr_i.
  \end{equation*}
  Clearly, $\bar\fc_0, \hat\fc_1, \ldots, \hat\fc_r$ depend smoothly on $\bp$ and $d$.

  \begin{step}
    We prove \eqref{Eq_ApproximationEstimate}.
  \end{step}
    
  We have
  \begin{equation}
    \label{Eq_TildeCBarCComparison}
    \bar L_{\bp,\epsilon} \bar\fc_\epsilon
    + Q_{\bp,\epsilon}(\bar\fc_\epsilon)
    - \bar L_{\bp,\epsilon} \tilde\fc
    - Q_{\bp,\epsilon}(\tilde\fc)
    =
      - \epsilon^{2k+2} \fr
  \end{equation}
  with $\fr = \fr_r$ as in the previous step.
  Both $\bar\fc$ and $\tilde\fc$ are small in $\bar X_\epsilon$;
  hence, it follows that
  \begin{equation*}
    \Abs{\bar\fc-\tilde\fc}_{\bar X_\epsilon} = O(\epsilon^{2k+2}).
  \end{equation*}
  To obtain estimates for the derivatives of $\bar\fc-\tilde\fc$,
  we differentiate \eqref{Eq_TildeCBarCComparison} and obtain an identity whose left-hand side is
  \begin{equation*}
    \bar L_{\bp,0} \nabla^m\del_\epsilon^n(\bar\fc-\tilde\fc)
      + 2Q_{\bp,0}\(\bar\fc, \nabla^m\del_\epsilon^n(\bar\fc-\tilde\fc)\)
      + 2Q_{\bp,0}\(\bar\fc-\tilde\fc, \nabla^m\del_\epsilon^n\tilde\fc\) 
  \end{equation*}
  and whose right-hand side can be controlled in terms of the lower order derivatives of $\hat\fd_\epsilon^k$.
  This gives the asserted estimates.  
\end{proof}

From \autoref{Prop_Approximation} it follows that $\fx$ is a homeomorphism onto its image and
that the estimate in \autoref{Thm_KuranishiModelNearZero}\itref{Thm_KuranishiModelNearZero_Expansion} holds with $\widehat\ob_i$ denoting the component of $\hat\fc_i$ in $\coker (\rd\fF)_{s_0}$.
This expansion implies that $\ob$ is $C^{2r-1}$ up to $\epsilon = 0$.
\qed


\section{Proof of \autoref{Thm_GenericOneParameterFamilies}}
\label{Sec_GenericOneParameterFamilies}

The first part of \autoref{Thm_GenericOneParameterFamilies} follows directly from \autoref{Thm_KuranishiModelNearZero}, since in this situation
\begin{equation*}
  \ob(\epsilon,t) = \dot\lambda(0) \cdot t + O(t^2) + O(\epsilon^2)
\end{equation*}
because $\ob_\del(t) = \dot\lambda(0) \cdot t + O(t^2)$.
The second part requires a more detailed analysis to show that
\begin{equation*}
  \ob(\epsilon,t) = \dot\lambda(0) \cdot t - \delta \epsilon^4 + O(t^2) + O(\epsilon^6).
\end{equation*}

To establish the above expansion of $\ob$, we solve
\begin{equation*}
  \bar L_{\epsilon}(o_\epsilon,\hat\fc) + Q_\epsilon(\hat\fc)
  +
  \begin{pmatrix}
    0 \\
    0 \\
    \epsilon^2 *\varpi F_{A_0} \\
    0
  \end{pmatrix} = 0
\end{equation*}
by formally expanding in $\epsilon^2$.
Inspection of \eqref{Eq_L0Inverse} shows that the obstruction to being able to solve $L_0\hat\fc = (\psi,b,\eta)$ is
\begin{equation*}
  -\pi\paren[\big]{\psi + \gamma\rII(\fa^*)^{-1}(b,\eta)}
\end{equation*}
where $\pi$ denotes the $L^2$--orthogonal projection onto $\ker\slD_\fH$.
In the case at hand, $\ker\slD_\fH = \R\Span{\Phi_0}$, and we have
\begin{align*}
  \inner{\Phi_0}{\gamma\rII^*(\fa^*)^{-1}(b,\eta)}_{L^2}
  &=
    \sum_{i=1}^3 \inner{\Phi_0}{\gamma(e_i)\nabla_{e_i}(\fa^*)^{-1}(b,\eta)}_{L^2} \\
  &=
    \sum_{i=1}^3 \inner{\gamma(e_i)\nabla_{e_i}\Phi_0}{(\fa^*)^{-1}(b,\eta)}_{L^2} = 0
\end{align*}
since $\fa \co \Omega^1(M,\fg_P) \oplus \Omega^0(M,\fg_P) \to \Gamma(\fN)$ and thus $(\fa^*)^{-1}$ also maps to $\Gamma(\fN)$.
Thus the obstruction reduces to
\begin{equation*}
  -\inner{\Phi_0}{\psi}_{L^2}.
\end{equation*}
By \eqref{Eq_L0Inverse}, the solution to $L_0(\phi,a,\xi) = (0,*\varpi F_{A_0}, 0)$ is
\begin{equation}
 \label{Eq_Obstruction1}
 \begin{split}
  \phi
  &=
    - \slD_\fH^{-1}\gamma\rII^*\chi
    - \chi,
  \qand \\ 
  (a,\xi)
  &=
    \fa^{-1}\slD_\fN\chi
    + \fa^{-1}\gamma\rII\slD_\fH^{-1}\gamma\rII^*\chi
  \end{split}
\end{equation}
with
\begin{equation}
  \label{Eq_Obstruction2}
  \chi \coloneq (\fa^*)^{-1}*\varpi F_{A_0}.
\end{equation}

Setting $\hat\fc_0 \coloneq \epsilon^2(\phi,a,\xi)$, we have
\begin{equation*}
  \epsilon^4\hat\fd_1
  \coloneq
    \bar L_\epsilon(0,\hat\fc_0) + Q_\epsilon(\hat\fc_0) + (0, 0, \epsilon^2 *\varpi F_{A_0}, 0)
  =
    O(\epsilon^4).
\end{equation*}
The component of $\hat\fd_1$ in $\Gamma(\fS)$ is
\begin{equation*}
  -\bar\gamma(a)\phi.
\end{equation*}
Using $\bar\gamma(a)\Phi_0 \in \Gamma(\fN)$ and $\rho(\fg_P)\Phi \perp \chi$,
we find that the obstruction to being able to solve $L_0(\phi_1,a_1,\xi_1) = \hat\fd_1$ is
\begin{equation*}
  \begin{split}
  \fo \coloneq \inner{\Phi_0}{\bar\gamma(a)\phi}_{L^2} 
  &=
    \inner{\bar\gamma(a)\Phi_0}{\phi}_{L^2} \\
  &=
    -\inner{\bar\gamma(a)\Phi_0}{\chi}_{L^2} \\
  &=
    -\inner{\fa(a,\xi)}{\chi}_{L^2} \\
  &=
    -\inner{\slD_\fN\chi
    + \gamma\rII\slD_\fH^{-1}\gamma\rII^*\chi}{\chi}_{L^2} \\
  &= 
    -\inner{\slD_\fN\chi}{\chi}_{L^2}
    +\inner{\slD_\fH^{-1}\gamma\rII^*\chi}{\gamma\rII^*\chi}_{L^2}.
   \end{split}
\end{equation*}
Comparing this with
\begin{align*}
  \inner{\slD_{A_0}\phi}{\phi}_{L^2}
  &= 
    \inner{\slD_{A_0}\slD_\fH^{-1}\gamma\rII^*\chi + \slD_{A_0}\chi}{\slD_\fH^{-1}\gamma\rII^*\chi+\chi}_{L^2} \\
  &= 
    \inner{(\slD_\fH+\gamma\rII)\slD_\fH^{-1}\gamma\rII^*\chi +(\slD_{\fN} - \gamma\rII^*)\chi}{\slD_\fH^{-1}\gamma\rII^*\chi+\chi}_{L^2} \\
  &= 
    \inner{\gamma\rII^*\chi}{\slD_\fH^{-1}\gamma\rII^*\chi}_{L^2} 
    +
    \inner{\gamma\rII\slD_\fH^{-1}\gamma\rII^*}{\chi}_{L^2} \\
  &\qquad
    +  
    \inner{\slD_\fN\chi}{\chi}_{L^2}
    -
    \inner{\gamma\rII^*\chi}{\slD_\fH^{-1}\gamma\rII^*\chi}_{L^2} \\
  &=   
    - \inner{\slD_\fH^{-1}\gamma\rII^*}{\gamma\rII^*\chi}_{L^2}
    +  
    \inner{\slD_\fN\chi}{\chi}_{L^2} \\
  &=
    -\fo
\end{align*}
completes the proof.
\qed


\appendix

\section{Examples of Seiberg--Witten equations}
\label{Sec_Examples}

\begin{example}
  \label{Ex_NonabelianMonopoles}
  Let $G = \U(n)$ and $S = \H \otimes_\C \C^n$, where the complex structure on $\H$ is given by right-multiplication by $i$.
  Let $\rho \co \U(n) \to \Sp(\H \otimes_\C \C^n)$ be induced from the standard representation of $\U(n)$. 
  The corresponding Seiberg--Witten equation is the \defined{$\U(n)$--monopole equation} in dimension three.
  The closely related $\PU(2)$--monopole equation on $4$--manifolds plays a crucial role in \citeauthor{Pidstrigach1995}'s approach to proving Witten's conjecture relating Donaldson and Seiberg--Witten invariants;
  see, e.g., \cite{Pidstrigach1995,Feehan1996,Teleman2000}.

  In this example as well as in \autoref{Ex_ClassicalSeibergWitten}, we have $\mu^{-1}(0) = \set{0}$.
\end{example}

\begin{example}
  \label{Ex_GCFlatness}
  Let $G$ be a compact Lie group, $\fg = \Lie(G)$, and fix an $\Ad$--invariant inner product on $\fg$.
  $S \coloneq \H\otimes_\R \fg$ is a quaternionic Hermitian vector space, and $\rho \co G \to \Sp(S)$ induced by the adjoint action is a quaternionic representation.
  The moment map $\mu\co \H\otimes_\R\fg \to (\Im\H\otimes\fg)^*$ is given by
  \begin{align*}
    \mu(\xi)
    &=
      \frac12[\xi,\xi] \\
    &=
      ([\xi_2,\xi_3]+[\xi_0,\xi_1])\otimes i
    + ([\xi_3,\xi_1]+[\xi_0,\xi_2])\otimes j
    + ([\xi_1,\xi_2]+[\xi_0,\xi_3])\otimes k
  \end{align*}
  for $\xi = \xi_0\otimes 1 + \xi_1\otimes i + \xi_2\otimes j + \xi_3\otimes k \in \H\otimes_\R\fg$.
  Set $H \coloneq \Sp(1) \times G$ and extend the above quaternionic representation of $G$ to $H$ by declaring that $q \in \Sp(1)$ acts by right-multiplication with $q^*$.

  Taking $Q$ to be the product of the chosen spin structure $\fs$ with a principal $G$--bundle, and choosing $B$ such that it induces the spin connection on $\fs$, \eqref{Eq_SeibergWitten} becomes
  \begin{align*}
    \rd_A^*a &= 0, \\
    *\rd_Aa + \rd_A\xi &= 0, \qand \\
    F_A &= \tfrac12[a\wedge a] + *[\xi,a].
  \end{align*}
  for $\xi \in \Gamma(\fg_P)$, $a \in \Omega^1(M,\fg_P)$ and $A \in \sA(P)$.
  If $M$ is closed, then integration by parts shows that every solution of this equation satisfies $\rd_A\xi=0$ and $[\xi,a]=0$; hence, $A + ia$ defines a \defined{flat $G^\C$--connection}.
  Here $G^\C$ denotes the complexification of $G$. 

  In the above situation, we have $\mu^{-1}(0)/G \iso (\H\otimes\ft)/W$ where $\ft$ is a the Lie algebra of a maximal torus $T \subset G$ and $W = N_G(T)/T$ is the Weyl group of $G$.
  However, since each $\xi \in \mu^{-1}(0)$ has stabilizer conjugate to $T$, we have $\mu^{-1}(0)\cap S^\reg = \emptyset$, and the hyperkähler quotient $S^\reg \hkred G$ is empty.  
\end{example}

\begin{example}
  \label{Ex_ADHMSeibergWitten}
  The motivating example for us is the \defined{$(r,k)$ ADHM Seiberg--Witten equation}, which we expect to play in important role in gauge theory on $\Gtwo$--manifolds,%
  \footnote{%
    More precisely, we expect solutions of the $(r,k)$ ADHM Seiberg--Witten equation to play a role in counter-acting the bubbling phenomenon along associative submanifolds discussed in \cite{Donaldson2009,Walpuski2013a};
    see also \cite{Haydys2017}.
  }
  and which arises from
  \begin{equation*}
    S = \Hom_\C(\C^r,\H\otimes_\C \C^k) \oplus \H^*\otimes_\R\fu(k)
  \end{equation*}
  with
  \begin{equation*}
    G = \U(k) \nsub H = \SU(r) \times \Sp(1) \times \U(k)
  \end{equation*}
  where $\SU(r)$ acts on $\C^r$ in the obvious way, $\U(k)$ acts on $\C^k$ in the obvious way and on $\fu(k)$ by the adjoint representation, and $\Sp(1)$ acts on the first copy of $\H$ trivially and on the second copy by right-multiplication with the conjugate.
  Accoding to \citet{Atiyah1978}, if $r \geq 2$, then $S^\reg \hkred G$ is the moduli space of framed $\SU(r)$ ASD instantons of charge $k$ on $\R^4$, and $\mu^{-1}(0)/G$ is its Uhlenbeck compactification, 
  If $r = 1$, then $\mu^{-1}(0)\cap S^\reg = \emptyset$, and $\mu^{-1}(0)/G = \Sym^k\H \coloneq \H^k/S_k$ by \citet[Example 3.14]{Nakajima1999}.
\end{example}


\section{Useful identities involving \texorpdfstring{$\mu$}{mu}}
\label{Sec_UsefulFormulae}

This appendix summarizes and proves a few useful identities regarding $\mu$, some of which are used in this article.

\begin{prop}
  \label{Prop_AMuXiMu}
  For $\xi \in \Omega^0(M,\fg_P)$, $a \in \Omega^1(M,\fg_P)$, and $\phi,\psi \in \Gamma(\fS)$, we have
  \begin{equation}
    \label{Eq_XiMu}
    [\xi,\mu(\phi,\psi)] = \mu(\phi,\rho(\xi)\psi) + \mu(\psi,\rho(\xi)\phi),
  \end{equation}
  and for $a \in \Omega^1(M,\fg_P)$ and $\phi,\psi \in \Gamma(\fS)$, we have
  \begin{equation}
    \label{Eq_AMu}
    2[a \wedge \mu(\phi,\psi)]
    =
      - * \rho^*\((\bar\gamma(a)\phi)\psi^*\)
      - * \rho^*\((\bar\gamma(a)\psi)\phi^*\).
  \end{equation}
\end{prop}

\begin{proof}
  For all $a \in \Omega^1(M,\fg_P)$, we have
  \begin{align*}
    2\inner{[\xi,\mu(\phi,\psi))]}{*a}
    &=
      \inner{\mu(\phi,\psi)}{-*[\xi,a]} \\
    &=
      \inner{\phi}{-\bar\gamma([\xi,a])\psi} \\
    &=
      -\inner{\phi}{\rho(\xi)\bar\gamma(a)\psi}
      +\inner{\phi}{\bar\gamma(a)\rho(\xi)\psi} \\
    &=
      \inner{\rho(\xi)\phi}{\bar\gamma(a)\psi}
      +\inner{\phi}{\bar\gamma(a)\rho(\xi)\psi} \\
    &=
      2\inner{\mu(\phi,\rho(\xi)\psi)}{*a}
      +2\inner{\mu(\psi,\rho(\xi)\phi)}{*a}.
  \end{align*}
  This proves the first identity.
  To prove the second identity, note that, for all $\eta \in \Omega^0(M,\fg_P)$, we have
  \begin{align*}
    2\inner{[a\wedge \mu(\phi,\psi)]}{*\eta}
    &=
      \inner{2\mu(\phi)}{*[\eta,a]} \\
    &=
      \inner{\phi}{\bar\gamma([\eta,a])\psi} \\
    &=
      \inner{\phi}{\rho(\xi)\bar\gamma(a)\psi}
      -\inner{\phi}{\bar\gamma(a)\rho(\xi)\psi} \\
    &=
      -\inner{\xi}{\rho^*\((\bar\gamma(a)\psi)\phi^*\)}
      -\inner{\xi}{\rho^*\((\bar\gamma(a)\phi)\psi^*\)}.
      \qedhere
  \end{align*}
\end{proof}

\begin{prop}
  \label{Prop_DMu_D*Mu}
  For all $A \in \sA(Q)$ and $\phi, \psi \in \Gamma(\fS)$ we have
  \begin{equation}
    \label{Eq_DMu}
    \rd_A\mu(\phi,\psi)
    = -*\frac12\rho^*\((\slD_A\phi)\psi^* + (\slD_A\psi)\phi^*\)
  \end{equation}
  and
  \begin{equation}
    \label{Eq_D*Mu}
    \begin{split}
      \rd_A^*\mu(\phi,\psi)
      &=
        *\mu(\slD_A\phi,\psi)
        + *\mu(\slD_A\psi,\phi) \\
      &\qquad
        - \frac12 \rho^*\((\nabla_A\phi)\psi^*\)
        - \frac12 \rho^*\((\nabla_A\psi)\phi^*\).
    \end{split}
  \end{equation}  
\end{prop}

\begin{proof}
  Fix a point $x \in M$, a positive local orthonormal frame $(e_i)$ around $x$ with $(\nabla e_i)(x) = 0$, and let $\xi$ be a local section of $\fg_P$ defined in a neighborhood of $x$ satisfying $(\nabla \xi)(x) = 0$.
  We set $\nabla^A_i \coloneq \nabla^A_{e_i}$.

  At the point $x \in M$, we compute with 
  \begin{align*}
     \inner{\rd_A\mu(\phi,\psi)}{*\xi}
    &=
      -\inner{\rd_A^**\mu(\phi,\psi))}{\xi} \\
    &=
      \frac12 \sum_{i=1}^3 
      \nabla^A_i \inner{\bar\gamma(\xi\otimes e^i) \phi}{\psi} \\
    &=
      \frac12 \(\inner{\rho(\xi)\slD_A \phi}{\psi} + \inner{\phi}{\rho(\xi)\slD_A\psi}\) \\
    &=
      - \frac12 \inner{\xi}{(\slD_A \phi)\psi^* + (\slD_A \psi)\phi^*}.
  \end{align*}
  This proves the first identity.
  To prove the second identity, we compute
  \begin{align*}
    \inner{\rd_A^*\mu(\phi,\psi)}{\xi}
    &=
      \inner{*~\rd_A* \mu(\phi,\psi)}{\xi} \\
    &=
      *\frac12 \sum_{i,j=1}^3 
      \nabla^A_i \inner{\bar\gamma(\xi\otimes e^j) \phi}{\psi} e^i\wedge e^j \\
    &=
      * \frac12 \sum_{i,j=1}^3 
      \(\inner{\bar\gamma(\xi\otimes e^j)\nabla^A_i \phi}{\psi}
      + \inner{\bar\gamma(\xi\otimes e^j)\nabla^A_i\psi}{\phi}\) e^i\wedge e^j \\
    &=
      \frac12 \sum_{i,j,k=1}^3 \epsilon_{ijk}^2 
      \Big(\inner{\rho(\xi)\gamma(e^k)\gamma(e^i)\nabla^A_i\phi}{\psi} \\
    &\qquad\qquad\qquad\qquad
      + \inner{\rho(\xi)\gamma(e^k)\gamma(e^i)\nabla^A_i\psi}{\phi}\Big) e^k \\
    &=
      \frac12 \sum_{k=1}^3
      \Big(\inner{\bar\gamma(\xi\otimes e^k) \slD_A \phi}{\psi}
      + \inner{\bar\gamma(\xi\otimes e^k) \slD_A \psi}{\phi} \\
    &\qquad\qquad\quad
      + \inner{\rho(\xi) \nabla^A_k \phi}{\psi}
      + \inner{\rho(\xi) \nabla^A_k \psi}{\phi}\Big) e^k \\
    &=
      \inner{\xi}{*\mu(\slD_A\phi,\psi)}
      + \inner{\xi}{*\mu(\slD_A\phi,\psi)} \\
    &\qquad
      + \frac12 \inner{\rho(\xi) \nabla_A\phi}{\psi}
      + \frac12 \inner{\rho(\xi) \nabla_A\psi}{\phi}.
      \qedhere
  \end{align*}
\end{proof}

\begin{prop}
  \label{Prop_Higgs}
  If $(\epsilon,\Phi,A) \in (0,\infty)\times\Gamma(\fS)\times \sA_B(Q)$ is a solution of \eqref{Eq_BlownUpSeibergWitten} and $R_\Phi(\xi) = \rho(\xi)\Phi$, then
  \begin{align*}
    (\rd_A^*\rd_A+\rd_A\rd_A^*+ \epsilon^{-2}R_\Phi^*R_\Phi)\mu(\Phi)
    &=
      \sum_{i,j=1}^3
        \frac12\rho^*\paren*{\paren[\big]{(F^B_{ij}+F^\fs_{ij}) \cdot \Phi}\Phi^*} e^{ij} \\
    &\qquad\qquad
      + \rho^*\paren*{(\nabla^A_j\Phi)(\nabla^A_i\Phi)^*} e^{ij}.
  \end{align*}
  Here $(e_1,e_2,e_3)$ is local orthonormal frame, $(e^1,e^2,e^3)$ is the dual coframe, $F^B_{ij} \coloneq F_B(e_i,e_j)$, $F^\fs_{ij} \coloneq F_\fs(e_i,e_j)$ with $F_\fs$ denoting the curvature of the spin connection on $\fs$, and $e^{ij} \coloneq e^i \wedge e^j$.
\end{prop}

\begin{proof}
  We compute
  \begin{align*}
    \rd_A \rho^*[(\nabla_A\Phi)\Phi^*]
    &=
      \sum_{i,j=1}^3 \rho^*[(\nabla^A_i\nabla^A_j\Phi)\Phi^*] e^{ij}
      + \rho^*[(\nabla^A_j\Phi)(\nabla^A_i\Phi)^*] e^{ij} \\
    &=
      \sum_{i,j=1}^3 \frac12\rho^*[(F^A_{ij} \cdot \Phi)\Phi^*] e^{ij}
      + \rho^*[(\nabla^A_j\Phi)(\nabla^A_i\Phi)^*] e^{ij}.
  \end{align*}
  Since
  \begin{equation*}
    \sum_{i,j=1}^3 \rho^*[\rho(\mu(\Phi)_{ij})\Phi)\Phi] e^{ij}
    = R_\Phi^*R_\Phi\mu(\Phi),
  \end{equation*}
  the result now follows from \autoref{Prop_DMu_D*Mu}.
\end{proof}


\section{Proof of \autoref{Prop_SWDGLA}}
\label{Sec_DGLA}

For the reader's convenience, we recall the definitions of the graded vector space $L^\bullet$,
  \begin{align*}
    L^0 &\coloneq \Omega^0(M,\fg_P), \\
    L^1 &\coloneq \Gamma(\fS)\oplus\Omega^1(M,\fg_P), \\
    L^2 &\coloneq \Gamma(\fS)\oplus\Omega^2(M,\fg_P), \qand \\
    L^3 &\coloneq \Omega^3(M,\fg_P),
  \end{align*}
  the graded Lie bracket $\LIE{\cdot}{\cdot}$,
  \begin{align*}
    \LIE{a}{b}
    &\coloneq
      \wlie{a}{b} && \text{for } a,b \in \Omega^\bullet(M,\fg_P), \\
    \LIE{\xi}{\phi}
    &\coloneq
      \rho(\xi)\phi &&\text{for } \xi \in \Omega^0(M,\fg_P) \text{ and } \phi \in \Gamma(\fS) \text{ in degree $1$ or $2$}, \\
    \LIE{a}{\phi}
    &\coloneq
      -\bar\gamma(a)\phi &&\text{for } a \in \Omega^1(M,\fg_P) \text{ and } \phi \in \Gamma(\fS) \text{ in degree $1$}, \\
    \LIE{\phi}{\psi}
    &\coloneq
      -2\mu(\phi,\psi) &&\text{for } \phi, \psi \in \Gamma(\fS) \text{ in degree $1$, and} \\
    \LIE{\phi}{\psi}
    &\coloneq
      -*\rho^*(\phi\psi^*) &&\text{for } \phi \in \Gamma(\fS) \text{ in degree $1$ and } \psi \in \Gamma(\fS) \text{ in degree 2},
  \end{align*}
  and the graded differential $\delta_{\fc}$,
  \begin{align*}
    \delta_\fc^0(\xi)
    &\coloneq
      \begin{pmatrix}
        -\rho(\xi)\Phi \\
        \rd_A\xi
      \end{pmatrix}, \\
    \delta_\fc^1(\phi,a)
    &\coloneq
      \begin{pmatrix}
        -\slD_A\phi - \bar\gamma(a)\Phi \\
        -2\mu(\Phi,\phi) + \rd_Aa
      \end{pmatrix}, \qand \\
    \delta_\fc^2(\psi,b)
    &\coloneq
      *\rho^*(\psi\Phi^* ) + \rd_Ab.
  \end{align*}

We proceed in four steps.

\setcounter{step}{0}
\begin{step}
  $(L^\bullet,\LIE{\cdot}{\cdot})$ is a graded Lie algebra.    
\end{step}

We need to verify the graded Jacobi identity, that is, for every three homogeneous elements $x,y,z \in L^\bullet$ we need to show that
\begin{equation*}
  J(x,y,z)
  \coloneq
  (-1)^{\deg{x}\cdot\deg{z}}\LIE{x}{\LIE{y}{z}}
  + (-1)^{\deg{y}\cdot\deg{x}}\LIE{y}{\LIE{z}{x}}
  + (-1)^{\deg{z}\cdot\deg{y}}\LIE{z}{\LIE{x}{y}}
\end{equation*}
vanishes.
Here $\deg{x}$ denotes the degree of $x$.

For degree reasons $J(x,y,z) = 0$, unless $\deg x + \deg y + \deg z \leq 3$.
$(\Omega^\bullet(M,\fg_P),{[\cdot\wedge\cdot]})$ is a graded Lie algebra.
Since $J(x,y,z)$ is invariant under permutations of $x$, $y$, and $z$, we can assume that $z \in \Gamma(\fS)$ in degree $1$ or $2$.
Hence, only the following five cases remain:
\begin{itemize}
\item 
  For $\xi,\eta \in \Omega^0(M,\fg_P)$, and $\phi \in \Gamma(\fS)$ in degree $1$ or $2$, we have
  \begin{align*}
    J(\xi,\eta,\phi)
    &=
      \LIE{\xi}{\LIE{\eta}{\phi}} + \LIE{\eta}{\LIE{\phi}{\xi}} + \LIE{\phi}{\LIE{\xi}{\eta}} \\
    &=
      \rho(\xi)\rho(\eta)\phi - \rho(\eta)\rho(\xi)\phi - \rho([\xi,\eta])\phi
      =
      0.
  \end{align*}
\item
  For $\xi \in \Omega^0(M,\fg_P)$, and $\phi,\psi \in \Gamma(\fS)$ in degree $1$, we have
  \begin{align*}
    J(\xi,\phi,\psi)
    &=
      \LIE{\xi}{\LIE{\phi}{\psi}} + \LIE{\phi}{\LIE{\psi}{\xi}} - \LIE{\psi}{\LIE{\xi}{\phi}} \\
    &=
      -2[\xi,\mu(\phi,\psi)] + 2\mu(\phi,\rho(\xi)\psi) + 2\mu(\psi,\rho(\xi)\phi)
      =
      0
  \end{align*}
  by \autoref{Prop_AMuXiMu}.
\item
  For $\xi \in \Omega^0(M,\fg_P)$, $\phi \in \Gamma(\fS)$ in degree $1$ and $\psi \in \Gamma(\fS)$ in degree $2$, we have
  \begin{align*}
    J(\xi,\phi,\psi)
    &=
      \LIE{\xi}{\LIE{\phi}{\psi}} + \LIE{\phi}{\LIE{\psi}{\xi}} + \LIE{\psi}{\LIE{\xi}{\phi}} \\
    &=
      -([\xi,*\rho^*\(\phi\psi^*\)]
      - *\rho^*\(\phi(\rho(\xi)\psi)^*\)
      + *\rho^*\(\psi(\rho(\xi)\phi)^*\)) \\
    &=
      -*\rho^*\([\rho(\xi),\phi\psi^*] + \phi\psi^*\rho(\xi) - \rho(\xi)\phi \psi^*\)
      =
      0.
  \end{align*}
\item
  For $\xi \in \Omega^0(M,\fg_P)$, $a \in \Omega^1(M,\fg_P)$, and $\phi \in \Gamma(\fS)$ in degree $1$, we have
  \begin{align*}
    J(\xi,a,\phi)
    &=
      \LIE{\xi}{\LIE{a}{\phi}} + \LIE{a}{\LIE{\phi}{\xi}} - \LIE{\phi}{\LIE{\xi}{a}}] \\
    &=
      -\rho(\xi)\bar\gamma(a)\phi + \bar\gamma(a)\rho(\xi)\phi + \bar\gamma(\lie{\xi}{a})\phi
      = 
      0.
  \end{align*}
\item
  For $a \in \Omega^1(M,\fg_P)$ and $\phi,\psi \in \Gamma(\fS)$ in degree $1$, we have
  \begin{align*}
    J(a,\phi,\psi)
    &=
      - \LIE{a}{\LIE{\phi}{\psi}} - \LIE{\phi}{\LIE{\psi}{a}} - \LIE{\psi}{\LIE{a}{\phi}} \\
    &=
      2\wlie{a}{\mu(\phi,\psi)}
      + *\rho^*\((\bar\gamma(a)\psi)\phi^*\)
      + *\rho^*\((\bar\gamma(a)\phi)\psi^*\)
      =
      0
  \end{align*}
  by \autoref{Prop_AMuXiMu}.

\end{itemize}

\begin{step}
  $(L^\bullet,\delta_\fc^\bullet)$ is a DGA.
\end{step}

We need to show that $\delta_\fc\circ \delta_\fc = 0$.
Using \autoref{Prop_AMuXiMu}, we compute that
\begin{align*}
  \delta_\fc^1\circ\delta_\fc^0(\xi)
  &=
    \begin{pmatrix}
      \slD_A \rho(\xi)\Phi - \bar\gamma(\rd_A\xi)\Phi \\
      2\mu(\Phi,\rho(\xi)\Phi) + \rd_A\rd_A\xi
    \end{pmatrix} \\
  &=
    \begin{pmatrix}
      \rho(\xi)\slD_A \Phi \\
      \lie{F_A-\mu(\Phi)}{\xi}
    \end{pmatrix} = 0,
\end{align*}
and, using \autoref{Prop_DMu_D*Mu} and \autoref{Prop_AMuXiMu}, we compute that
\begin{align*}
  \delta_\fc^2\circ\delta_\fc^1(\phi,a)
  &=
    -*\rho^*\((\slD_A\phi)\Phi^*\) - *\rho^*\((\bar\gamma(a)\Phi)\Phi^*\)
    - 2\rd_A\mu(\Phi,\phi) + \rd_A\rd_A a \\
  &=
    \rho^*\((\slD_A\Phi)\phi\)
    + \wlie{(F_A-\mu(\Phi))}{a} = 0.
\end{align*}

\begin{step}
  $(L^\bullet,\LIE{\cdot}{\cdot},\delta_\fc^\bullet)$ is a DGLA.
\end{step}

We need to verify that $\delta_\fc^\bullet$ satisfies the graded Leibniz rule with respect to $\LIE{\cdot}{\cdot}$, that is for every two homogeneous elements $x,y \in L^\bullet$ we need to show that
\begin{equation*}
  D(x,y) 
  =
  \delta \LIE{x}{y}
  - \LIE{\delta x}{y}
  - (-1)^{\deg x}\LIE{x}{\delta y}
\end{equation*}
vanishes.

For degree reasons, $D(x,y) = 0$ unless $\deg x + \deg y \leq 2$;
hence, only the following eight cases remain:
\begin{itemize}
\item
  For $\xi,\eta \in \Omega^0(M,\fg_P)$, we have
  \begin{equation*}
    D(\xi,\eta)
    =
    \begin{pmatrix}
      -\rho([\xi,\eta])\Phi \\
      \rd_A[\xi,\eta]
    \end{pmatrix}
    -
    \LIE{
      \begin{pmatrix}
        -\rho(\xi)\Phi \\
        \rd_A\xi
      \end{pmatrix}
    }{
      \eta
    }
    -
    \LIE{
      \xi
    }{
      \begin{pmatrix}
        -\rho(\eta)\Phi \\
        \rd_A\eta
      \end{pmatrix}
    } = 0.
  \end{equation*}
\item
  For $\xi \in \Omega^0(M,\fg_P)$ and $\phi \in \Gamma(\fS)$ in degree $1$, we have
  \begin{align*}
    D(\xi,\phi)
    &=
      \begin{pmatrix}
        -\slD_A\rho(\xi)\phi \\
        -2\mu(\Phi,\rho(\xi)\phi)
      \end{pmatrix}
    -
    \LIE{
    \begin{pmatrix}
      -\rho(\xi)\Phi \\
      \rd_A\xi
    \end{pmatrix}
    }{
    \phi
    }
    -
    \LIE{
    \xi
    }{
    \begin{pmatrix}
      -\slD_A\phi \\
      -2\mu(\Phi,\phi)
    \end{pmatrix}
    } \\
    &= 
      \begin{pmatrix}
        -\slD_A\rho(\xi)\phi + \bar\gamma(\rd_A\xi)\phi + \rho(\xi)\slD_A\phi \\
        -2\mu(\Phi,\rho(\xi)\phi) - 2\mu(\rho(\xi)\Phi,\phi) + 2[\xi,\mu(\Phi,\phi)]
      \end{pmatrix}
    =
    0
  \end{align*}
  by \autoref{Prop_AMuXiMu}.
\item
  For $\xi \in \Omega^0(M,\fg_P)$ and $a \in \Omega^1(M,\fg_P)$, we have
  \begin{align*}
    D(\xi,a)
    &=
      \begin{pmatrix}
        -\bar\gamma([\xi,a])\Phi \\
        \rd_A [\xi,a]
      \end{pmatrix}
    -
    \LIE{
    \begin{pmatrix}
      -\rho(\xi)\Phi \\
      \rd_A\xi
    \end{pmatrix}
    }{
    a
    }
    -
    \LIE{
    \xi
    }{
    \begin{pmatrix}
      -\bar\gamma(a)\Phi \\
      \rd_A a
    \end{pmatrix}
    } \\
    &=
      \begin{pmatrix}
        -\bar\gamma([\xi,a])\Phi - \bar\gamma(a)\rho(\xi)\Phi + \rho(\xi)\bar\gamma(a)\Phi \\
        \rd_A [\xi,a] - [\rd_A\xi \wedge a] - [\xi,\rd_A]
      \end{pmatrix}
    =
    0.
  \end{align*}
\item
  For $\xi \in \Omega^0(M,\fg_P)$ and $\phi \in \Gamma(\fS)$ in degree $2$, we have
  \begin{align*}
    D(\xi,\phi)
    &=
      *\rho^*(\rho(\xi)\phi\Phi^*)
      -
     \LIE{\rho(\xi)\Phi}{\phi}
      -
      \LIE{\rd_A\xi}{\phi}
      -
      \LIE{
      \xi
      }{
      *\rho^*(\phi\Phi^*)
      } \\
    &=
      *\rho^*(\rho(\xi)\phi\Phi^*)
      - * \rho^*( \phi \Phi^* \rho(\xi))
      - [\xi,*\rho^*(\phi\Phi^*)] = 0.
  \end{align*}
\item
  For $\xi \in \Omega^0(M,\fg_P)$ and $b \in \Omega^2(M,\fg_P)$, we have
  \begin{equation*}
    D(\xi,b)
    = \rd_A [\xi,b] - [\rd_A\xi,b] - [\xi,\rd_Ab] = 0.
  \end{equation*}
\item
  For $\phi,\psi \in \Gamma(\fS)$ in degree $1$, we have
  \begin{align*}
    D(\phi,\psi)
    &=
      -2\rd_A\mu(\phi,\psi)
      -
      \LIE{
      \begin{pmatrix}
        -\slD_A\phi \\
        -2\mu(\Phi,\phi)
      \end{pmatrix}
    }{
    \psi
    }
    +
    \LIE{
    \phi
    }{
    \begin{pmatrix}
      -\slD_A\psi \\
      -2\mu(\Phi,\psi)
    \end{pmatrix}
    } \\
    &=
      -2\rd_A\mu(\phi,\psi) - *\rho^*\((\slD_A\phi)\psi^*\) - *\rho^*\((\slD_A\psi)\phi^*\)
      =
      0
  \end{align*}
  by \autoref{Prop_DMu_D*Mu}.
\item
  For $a \in \Omega^1(M,\fg_P)$ and $\phi \in \Gamma(\fS)$ in degree $1$, we have
  \begin{align*}
    D(a,\phi)
    &=
      *\rho^*\((\bar\gamma(a)\phi)\Phi^*\)
      -
      \LIE{
      \begin{pmatrix}
        -\bar\gamma(a)\Phi \\
        \rd_A a
      \end{pmatrix}
    }{
    \phi 
    }
    +
    \LIE{
    a
    }{
    \begin{pmatrix}
      -\slD_A\phi \\
      -2\mu(\Phi,\phi)
    \end{pmatrix}
    } \\
    &=
      -*\rho\((\bar\gamma(a)\phi)\Phi^*\)
      -*\rho^*\(\bar\gamma(a)\Phi)\phi^*\)
      -2\wlie{a}{\mu(\Phi,\phi)}
      =
      0
  \end{align*} 
  by \autoref{Prop_AMuXiMu}.
\item
  For $a,b \in \Omega^1(M,\fg_P)$, we have
  \begin{equation*}
    D(a,b)
    =
    \begin{pmatrix}
    - \bar\gamma{ \wlie{a}{b}} \Phi \\
    \rd_A \wlie{a}{b}
    \end{pmatrix}
    -
    \LIE{
      \begin{pmatrix}
        \bar\gamma(a)\Phi \\
        \rd_A a
      \end{pmatrix}
    }{
      b
    }
    +
    \LIE{
      a
    }{
      \begin{pmatrix}
        \bar\gamma(b)\Phi \\
        \rd_A b
      \end{pmatrix}
    }
    =
    0.
  \end{equation*}
\end{itemize}

\begin{step}
  For every $\hat\fc \coloneq (a,\phi) \in L^1$, $(A+a,\Phi+\phi)$ solves \eqref{Eq_SeibergWitten} if and only if $\delta_\fc \hat\fc + \frac12\LIE{\hat\fc}{\hat\fc} = 0$.
\end{step}

For $\hat\fc \coloneq (a,\phi) \in L^1$, we have
\begin{equation*}
  \delta\fc + \frac12\LIE{\fc}{\fc}
  =
  \begin{pmatrix}
    -\slD_A\phi - \bar\gamma(a)\Phi - \bar\gamma(a)\phi \\
    -2\mu(\Phi,\phi) - \mu(\phi,\phi) + \rd_A a + \frac12[a\wedge a]
  \end{pmatrix},
\end{equation*}
which vanishes if and only if $(A+a,\Phi+a)$ solves \eqref{Eq_SeibergWitten}.
\qed


\printreferences

\end{document}
